\documentclass[11pt,letterpaper]{amsart}
\usepackage{verbatim}
\usepackage{lscape}
\usepackage{tabularx}
\usepackage{fullpage}
\usepackage{appendix}
\usepackage{amsmath}

\usepackage{upgreek}
\usepackage{tikz-cd}
\usepackage{enumitem}
\usepackage{amsthm}
\usepackage{extarrows}
\usepackage{xypic}
\usepackage{stmaryrd}
\usepackage{graphicx}
\usepackage{float}
\usepackage{graphics,amssymb}
\usepackage{amsfonts}
\usepackage{latexsym}
\usepackage{epsf}
\usepackage{todonotes}
\usepackage{mathrsfs}
\usepackage{bbm}
\usepackage{hyperref}
\hypersetup{
    colorlinks=true,
    linkcolor=blue,
    urlcolor=blue, 
    citecolor= blue}
\usepackage[noabbrev, capitalize, nameinlink]{cleveref}
\usepackage{thmtools}
\usepackage{thm-restate}
\usepackage[framemethod=tikz]{mdframed}
\usepackage{ulem}
\usepackage{cancel}
\usepackage{tikz}
\usetikzlibrary{matrix,arrows}
\usepackage{mathtools}
\usepackage{url}

\usepackage{fancybox}
\usepackage{dsfont}
\usepackage{bbding}
\usepackage{setspace}
\usepackage{booktabs}
\usepackage{spectralsequences}
\usepackage{array}
\usetikzlibrary{fit}

\newtheorem{theoremA}{Theorem}

\crefname{theoremA}{Theorem}{Theorems}
\crefname{propositionA}{Proposition}{Propositions}

\newtheorem{propositionA}[theoremA]{Proposition}
\newtheorem{theorem}{Theorem}[section]
\newtheorem{lemma}[theorem]{Lemma}

\newtheorem{lem/def}[theorem]{Lemma/Definition}
\newtheorem{proposition}[theorem]{Proposition}
\newtheorem{corollary}[theorem]{Corollary}

\newtheorem*{theorem*}{Theorem}
\newtheorem*{corollary*}{Corollary}
\newtheorem*{proposition*}{Proposition}

\theoremstyle{definition}
\newtheorem{remark}[theorem]{Remark}
\newtheorem*{remark*}{Remark}
\newtheorem{definition}[theorem]{Definition}
\newtheorem*{definition*}{Definition}

\newcommand{\TT}{\mathbb{T}} 

\newcommand{\Gal}[2]{\mathrm{Gal}(#1/#2)}

\newcommand{\Z}{\mathbb{Z}}
\newcommand{\Q}{\mathbb{Q}}

\newcommand{\m}{\mathfrak{m}}

\newcommand{\pp}{\mathfrak{p}}
\newcommand{\F}{\mathbb{F}}

\newcommand{\lb}{\llbracket}
\newcommand{\rb}{\rrbracket}

\DeclareMathOperator{\ab}{\mathrm{ab}}

\DeclareMathOperator{\End}{\mathrm{End}}

\DeclareMathOperator{\Frob}{\mathrm{Frob}}
\DeclareMathOperator{\GL}{\mathrm{GL}}

\DeclareMathOperator{\im}{\mathrm{Im}}
\DeclareMathOperator{\Ind}{\mathrm{Ind}}

\DeclareMathOperator{\PGL}{\mathrm{PGL}}
\DeclareMathOperator{\prim}{\mathrm{prim}}

\DeclareMathOperator{\rk}{\mathrm{rk}}

\DeclareMathOperator{\SL}{SL}

\DeclareMathOperator{\tr}{\mathrm{tr}}
\DeclareMathOperator{\univ}{univ}

\newcommand{\II}{\mathfrak{I}}

\title{Modularity theorems for Eisenstein congruences in prime-power level}
\author{Jaclyn Lang, Katharina Müller, Bharathwaj Palvannan}
\allowdisplaybreaks

\begin{document}

\title{Modularity theorems for Eisenstein congruences in prime-power level}

\begin{abstract}
Let $p,N\geq 5$ be primes such that $N \equiv 1 \bmod p$. We prove modularity theorems at levels $N$ and $N^2$, showing that suitable Eisenstein localizations of the weight-$2$ $p$-adic Hecke algebra at these levels are isomorphic to certain natural quotients of a universal pseudodeformation ring. This universal ring parametrizes pseudorepresentations that are residually Eisenstein, unramified outside $Np$ and finite-flat at $p$, and satisfy appropriate conditions at $N$ depending on the level.
\end{abstract}





\maketitle
\section{Introduction}\label{sec:introduction}

Let $p,N \geq 5$ be primes such that $N \equiv 1 \bmod p$.  The question of studying the $\Z_p$-rank $r$ of Mazur's Eisenstein Hecke algebra in weight $2$ and prime level $\Gamma_0(N)$ has received considerable interest in literature \cite{Merel,CE2005,Lecouturier,WWE}.  When $r-1$ equals $1$ or $2$, it equals the order of vanishing of the mod-$p$ reduction of a zeta element that interpolates Dirichlet $L$-values at $-1$. A uniform proof of this fact was provided by Lecouturier \cite{Lecouturier} --- substantially expanding  on earlier work of Merel \cite{Merel} --- using the theory of higher Eisenstein elements.  A key motivation in Lecouturier's work is a consideration of suitable combinations of Eisenstein series over $\Gamma_1(N)$. More recently, the authors gave an independent proof of this same fact by studying Eisenstein--cuspidal congruences at level $\Gamma_0(N^2)$. This paper seeks to formalize the underlying connections between the two works by proving that the modularity theorems that quantify Eisenstein--cuspidal congruences at levels $\Gamma_1(N)$ and $\Gamma_0(N^2)$ are, in fact, equivalent.  Since modularity theorems at level $\Gamma_0(N)$ have  been successfully employed to study the rank of Mazur's Eisenstein Hecke algebra \cite{CE2005,WWE}, another important source of motivation is to seek their generalizations at prime-power levels.

As in \cite{WWE}, we consider modularity theorems via the pseudodeformation theory of the representation $\overline{D} \coloneqq \omega \oplus 1$, where $\omega$ is the mod-$p$ cyclotomic character. We consider four pseudodeformation rings, suggestively denoted $R_{\Gamma_1(N^2)}$, $R_{\Gamma_0(N^2)}$, $R_{\Gamma_1(N)}$ and $R_{\Gamma_0(N)}$. Each ring parametrizes those pseudodeformations that lift $\overline{D}$ with additional properties prescribed to match those of weight two modular forms at the indicated level.  In particular, the parametrized pseudodeformations are always unramified outside $Np$ and finite flat at $p$; the only differences are a possible determinant condition and a local condition at $N$.  We refer the reader to \cref{sec:DefRings} for their precise definitions. The deformation ring $R_{\Gamma_0(N)}$ was considered by Wake--Wang-Erickson \cite{WWE} in their study of the $\Z_p$-rank of Mazur's Eisenstein Hecke algebra at level $\Gamma_0(N)$ using Massey products in Galois cohomology.  On the Hecke side, we consider appropriate Eisenstein localizations of weight-$2$ Hecke algebras of levels $\Gamma_1(N^2)$, $\Gamma_0(N^2)$,  $\Gamma_1(N)$ and $\Gamma_0(N)$. We denote them by $\TT_{\Gamma_1(N^2)}$,  $\TT_{\Gamma_0(N^2)}$,  $\TT_{\Gamma_1(N)}$ and $\TT_{\Gamma_0(N)}$, respectively. For each $\Gamma \in \{\Gamma_1(N^2), \Gamma_0(N^2), \Gamma_1(N), \Gamma_0(N)\}$, there is a natural ring morphism $R_\Gamma \rightarrow \TT_\Gamma$.  Let $p^s$ denote the higher power of $p$ dividing $N-1$ and $r \coloneqq \rk_{\Z_p} \TT_{\Gamma_0(N)}$.  Our first main theorem is the following. 

\begin{theoremA}\label{thmA:modularity}
The following statements are equivalent:
\begin{enumerate}[label=(\roman*),leftmargin=*]
    \item\label{thmpart:Gamma0N2+rank} $R_{\Gamma_0(N^2)} \cong \TT_{\Gamma_0(N^2)}$ and $\rk_{\Z_p} \TT_{\Gamma_0(N^2)} = \bigl(\frac{p^s + 1}{2}\bigr)r$;
    \item\label{thmpart:Gamma1N+rank}   $R_{\Gamma_1(N)} \cong \TT_{\Gamma_1(N)}$  and $\rk_{\Z_p} \TT_{\Gamma_1(N)} = p^sr$.
    \end{enumerate}
\end{theoremA}

Moreover, the equivalent statements in \cref{thmA:modularity} are true, which we establish by proving \cref{thmA:modularity}\ref{thmpart:Gamma0N2+rank}. Namely, a recent freeness result of the first author with Pollack and Wake \cite{LangPollackWake} together with an application of Nakayama's lemma allows us to deduce \cref{thmA:modularity}\ref{thmpart:Gamma0N2+rank} from the modularity theorem of Wake--Wang-Erickson \cite{WWE} over $\Gamma_0(N)$.

\begin{theoremA}\label{thmA:R=TGamma0N2}
 \cref{thmA:modularity}\ref{thmpart:Gamma0N2+rank} holds.  That is, $R_{\Gamma_0(N^2)} \cong \TT_{\Gamma_0(N^2)}$  and  $\rk_{\Z_p} \TT_{\Gamma_0(N^2)} = \bigl(\frac{p^s + 1}{2}\bigr)r$.
\end{theoremA}

Combining \cref{thmA:R=TGamma0N2} with a freeness result for $\TT_{\Gamma_1(N^2)}$ over the $\Z_p$-algebra generated by the diamond operators, a further application of Nakayama's lemma  yields the following modularity theorem at level $\Gamma_1(N^2)$.

\begin{theoremA} \label{thmA:Gamma1N2+rank} 
 We have $R_{\Gamma_1(N^2)} \cong \TT_{\Gamma_1(N^2)}$  and $\rk_{\Z_p} \TT_{\Gamma_1(N^2)} = \bigl(\frac{p^s(p^s + 1)}{2}\bigr)r$.
\end{theoremA}

The crucial input to our proof of \cref{thmA:R=TGamma0N2} --- the freeness result of \cite{LangPollackWake}  --- is established via techniques that are available beyond ${\GL_2}_{/\Q}$. Let $\Delta$ be the maximal $p$-quotient of $(\Z/N\Z)^\times$, and let $\Z_p[\Delta]^+$ be subring of $\Z_p[\Delta]$ that is fixed by the involution that inverts group-like elements. The ring $\Z_p[\Delta]^+$ acts as a natural substitute for the ring generated by diamond operators which are absent in level $\Gamma_0(N^2)$. The action of the ring $\Z_p[\Delta]^+$ on $\TT_{\Gamma_0(N^2)}$ arises naturally in two different ways. First, it arises as an appropriate inertia-at-$N$ pseudodeformation ring, making $R_{\Gamma_0(N^2)}$, and hence $\TT_{\Gamma_0(N^2)}$, an algebra over it. This action of $\Z_p[\Delta]^+$ on $\TT_{\Gamma_0(N^2)}$ is referred as the \textit{Galois action}.  The ring  $\Z_p[\Delta]^+$ also arises as the endomorphism algebra of a projective indecomposable representation of $\Z_p[\PGL_2(\F_N)]$, which leads to a second action of $\Z_p[\Delta]^+$ on $\TT_{\Gamma_0(N^2)}$, called the \textit{automorphic action}. Using the modular representation theory of $\PGL_2(\F_N)$,  the authors of \cite{LangPollackWake} show that, when equipped with the automorphic action, $\TT_{\Gamma_0(N^2)}$ is free over $\Z_p[\Delta]^+$.  That the Galois and automorphic actions agree is established in \cite{LangPollackWake} as an instance of local Langlands in families.  

One can also give a direct proof of \cref{thmA:modularity}\ref{thmpart:Gamma1N+rank}. The strategy for its proof follows the deduction of \cref{thmA:Gamma1N2+rank} from \cref{thmA:R=TGamma0N2} using Nakayama's lemma.  As before the crucial tool is a freeness result. The ring $\TT_{\Gamma_1(N)}$ turns out to be free of the appropriate rank over the $\Z_p[\Delta]$ -- the $\Z_p$-algebra generated by the diamond operators. This fact was first observed in an unpublished work of Lecouturier--Wake--Wang-Erickson; see \cref{rem:LWWE}. 

The crux of the equivalence in \cref{thmA:modularity} is based on the fact that forms in level $\Gamma_0(N^2)$ and $\Gamma_1(N)$ are related by twisting by $p$-power Dirichlet characters. Twisting arguments can be leveraged to show that the crucial freeness results at levels $\Gamma_0(N^2)$ and $\Gamma_1(N)$, which are input to the corresponding modularity theorems, are equivalent. 
\begin{propositionA}\label{prop:freenessequiv}
    The following statements are equivalent:
    \begin{enumerate}[leftmargin=*]
    \item The $\Z_p[\Delta]^+$-module $\TT_{\Gamma_0(N^2)}$ is free of rank $r$.
    \item The $\Z_p[\Delta]$-module $\TT_{\Gamma_1(N)}$ is free of rank $r$.
    \end{enumerate}
\end{propositionA}

One can approach questions of modularity in the residually reducible case by the techniques developed by Skinner and Wiles \cite{SkinnerWiles99}.  However, we believe their techniques would only show that the natural map $R_{\Gamma} \to \TT_{\Gamma}$ is an isomorphism after taking the quotient by the nilradical and inverting $p$ on both rings, neither of which we require for our results. We expect our precise integral modularity theorems to be of greater utility. Indeed, the integral modularity theorem at level $\Gamma_0(N)$ was important both in works of Calegari--Emerton \cite{CE2005} and Wake--Wang-Erickson \cite{WWE}.  In forthcoming work, we establish a connection between the modularity theorem $R_{\Gamma_0(N^2)} \cong \TT_{\Gamma_0(N^2)}$ and the equivariant Tamagawa number conjecture at $s = -1$ for the Tate motive over the maximal $p$-extension of $\Q(\zeta_N)/\Q$. Here, $\zeta_N$ denotes a primitive $N$-th root of unity. The integral modularity theorem is essential for establishing this connection.

\subsection{Leitfaden}
\cref{sec:HeckeAlgsDefRings} contains definitions and preliminary results on Eisenstein series and Hecke algebras. \cref{sec:DefRings} introduces the pseudodeformation rings that will be matched with the corresponding Hecke algebras. \cref{sec:reductionstep} contains the abstract commutative lemmas and counting arguments for ranks of Hecke algebras that eventually allow us to leverage information from level $\Gamma_0(N)$. \cref{sec:modrepthy} introduces the $\Z_p[\Delta]^+$-algebra structure on $\TT_{\Gamma_0(N^2)}$ via its Galois action and completes the proofs of the main results stated in \cref{sec:introduction}.

\subsection{Notation}\label{subsec:notation}
Throughout we fix primes $N, p \geq 5$ such that $N \equiv 1 \bmod p$.  For any integer $n \geq 1$ we write $\zeta_n$ for a primitive $n$-th root of unity.  Let $K/\Q$ be the maximal subfield of $\Q(\zeta_N)$ with $p$-power degree, which is nontrivial since $N \equiv 1 \bmod p$.  Let $\Delta = \Gal{K}{\Q}$, which is canonically identified with the maximal $p$-power quotient of $(\Z/N\Z)^\times$.  Let $p^s \coloneqq |\Delta|$.  For $\delta \in \Delta$, we use $[\delta]$ to denote the corresponding group-like element in the group ring $\Z_p[\Delta]$.

For any field $L$ of characteristic zero, fix an algebraic closure $\overline{L}$ of $L$, and let $G_L = \Gal{\overline{L}}{L}$.  For each prime $\ell$ we fix embeddings $\overline{\Q} \hookrightarrow \overline{\Q}_\ell$, which induces a map $G_{\Q_\ell} \to G_\Q$.  Write $I_\ell$ for the inertia subgroup of $G_{\Q_\ell}$, which we may view inside $G_\Q$ via the given map.  We fix the set $S = \{N, p, \infty\}$ and write $G_{\Q,S}$ for the quotient of $G_\Q$ corresponding to the maximal extension of $\Q$ unramified outside $S$.  For each prime $\ell \not\in S$, fix a choice of a Frobenius element at $\ell$, denoted $\Frob_\ell \in G_{\Q,S}$.  Let $\varepsilon \colon G_{\Q, S} \to \Z_p^\times$ be the $p$-adic cyclotomic character and $\omega \colon G_{\Q, S} \to \F_p^\times$ its reduction modulo $p$.  Let $\kappa \colon G_{\Q, S} \to \Z_p[\Delta]^\times$ be the natural projection of $G_{\Q, S}$ onto $\Delta$ composed with the natural inclusion into $\Z_p[\Delta]$.  Given a finite order character $\chi$, we write $\Z_p[\chi]$ for the $p$-adic ring (topologically) generated by the values of $\chi$.

For an integer $M$, let $\Gamma_0(M)$ be the subgroup of $\SL_2(\Z)$ whose lower left entry is divisible by $M$ and $\Gamma_1(M)$ the subgroup of $\Gamma_0(M)$ whose diagonal entries are congruent to $1$ modulo $M$.  We use $\Gamma$ to denote any of the $\Gamma_0(N), \Gamma_1(N), \Gamma_0(N^2), \Gamma_1(N^2)$.  For a $\Z_p$-algebra $A$, write $M_2(\Gamma; A)$ for the space of weight $2$ modular forms of level $\Gamma$ with coefficients in $A$ and $S_2(\Gamma; A)$ for the subspace of cuspforms.  When $A = \Z_p$ we write $M_2(\Gamma)$ and $S_2(\Gamma)$.  Recall that $M_2(\Gamma; A) = M_2(\Gamma) \otimes A$ and similarly for $S_2(\Gamma; A)$.  

For each prime $\ell$ we have a Hecke operator $T_\ell$ acting on $M_2(\Gamma; A)$. 
When $\Gamma = \Gamma_1(M)$, there is a natural action on $M_2(\Gamma, A)$ by $(\Z/M\Z)^\times$ given by the diamond operators $\langle d \rangle$ for $d \in (\Z/M\Z)^\times$.  When $\varphi(M) \coloneqq |(\Z/M\Z)^\times|$ is invertible in $A$ and $A$ contains the $\varphi(M)$-th roots of unity, there is a natural decomposition into eigenspaces for this action,
\[
M_2(\Gamma_1(M); A) = \bigoplus_{\chi \colon (\Z/M\Z)^\times \to A^\times} M_2(M, \chi; A),
\]
where $M_2(M, \chi; A)$ is the $\chi$-isotypic component.  A similar decomposition for cusp forms, and we use $S_2(M, \chi; A)$ for the $\chi$-isotypic component of $S_2(\Gamma_1(M); A)$.  For $f \in M_2(M, \chi; A)$, we call $\chi$ the \textit{Nebentypus} of $f$ and denote it by $\chi_f$.

We call a modular form $f = \sum_{n=0}^\infty a_n(f)q^n$ \textit{normalized} if $a_1(f) = 1$.  A normalized form $f \in S_2(\Gamma_1(M))$ is \textit{primitive} if it is an eigenform for all Hecke operators away from $M$ and $f$ is a newform of level $M$ \cite[p. 164]{Miyakebook}.  

Let $\overline{\Z}_p$ be the integral closure of $\Z_p$ in $\overline{\Q}_p$ and $\pp$ its unique maximal ideal.  We say that $f, g \in M_2(\Gamma; \overline{\Z}_p)$ are \textit{congruent modulo $\pp$} if $a_n(f) \equiv a_n(g) \bmod \pp$ for all $n \neq 0$.  We say that $f$ and $g$ are \textit{congruent modulo $\pp$ away from $N$} if $a_n(f) \equiv a_n(g) \bmod \pp$ for all $n \geq 0$ such that $N \nmid n$.  

If $f \in S_2(N^r, \chi; \overline{\Q}_p)$ is a normalized Hecke eigenform for the operators away from $N$, write $\rho_f$ for the $\pp$-adic Galois representation associated with $f$ satisfying $\tr \rho_f(\Frob_\ell) = a_\ell(f)$ and $\det \rho_f(\Frob_\ell) = \chi(\ell)\ell$ for all primes $\ell \nmid Np$.


\section{Hecke algebras}\label{sec:HeckeAlgsDefRings}
In this section we define the relevant Hecke algebras and deformation rings that are needed to state the main theorems in \cref{sec:introduction} precisely.  We begin with a review of Eisenstein series in \cref{subsec:EisSeries} and then define the Hecke algebras that parametrize Eisenstein congruences in \cref{subsec:Heckealgs}.  

\subsection{Eisenstein series}
\label{subsec:EisSeries}
Recall the function
\[
E_2(z) \coloneqq \frac{-1}{24} + \sum_{n \geq 1} \sigma(n)q^n,
\]
where $q = e^{2\pi iz}$ and $\sigma(n) \coloneqq \sum_{0 < d \mid n} d$.  It is a nearly holomorphic modular form of weight 2 and level 1.  One obtains a modular form of level $\Gamma_0(N)$ with $N$ prime via $N$-stabilization. Explicitly,
\[
E_{2,N}(z) \coloneqq E_2(z) - NE_2(Nz) \in M_2(\Gamma_0(N)).
\]
It is an eigenform with $T_\ell$-eigenvalue $\ell+1$ for all primes $\ell \neq N$ and $T_N$-eigenvalue 1.  Much of this article is concerned with understanding congruences between $E_{2,N}$ and other modular forms of weight 2 and level $\Gamma_1(N^j)$ for some $j \geq 1$.  It is sufficient to consider congruences with newforms.  Carayol showed that one need only consider level $\Gamma_1(N^2)$; more precisely, if $f$ is a primitive eigenform of weight 2 and level $\Gamma_1(N^j)$ for some $j \geq 1$ that is congruent to $E_{2,N}$ modulo $\pp$ away from $N$, then $j \leq 2$ \cite[\S 1.1]{Carayol89}.  In this section we describe the Eisenstein series of weight 2 and level $\Gamma_1(N^2)$ that are Hecke eigenforms and are congruent to $E_{2,N}$ modulo $\pp$ away from $N$.  This is standard and can be found in most introductory textbooks on modular forms, for instance \cite[Theorem 4.6.2]{DiamondShurman}.

\subsubsection{Level \texorpdfstring{$\Gamma_1(N)$}{}}\label{level1N}
It is well known that $E_{2,N}$ is the only Eisenstein series of weight 2 and level $\Gamma_0(N)$.  To define the others of level $\Gamma_1(N)$, we establish some notation.  If $\eta, \varphi$ are Dirichlet characters and $n > 0$ an integer, define
\[
\sigma_{\eta,\varphi}(n) \coloneqq \sum_{0 < d \mid n} \eta(n/d)\varphi(d)d.
\]
When $\eta$ and $\varphi$ are both trivial, we recover the function $\sigma$ that describes the Fourier coefficients of $E_2$ above.  For any nontrivial Dirichlet character $\chi$ of conductor $N$ such that $\chi(-1) = 1$, there is a two-dimensional space of Eisenstein series in $M_2(N, \chi; \overline{\Q}_p)$.  It is spanned by $E_{\chi, 1}$ and $E_{1, \chi}$, whose $q$-expansions at infinity are given by
\[
E_{\chi,1}(z) \coloneqq \sum_{n \geq 1} \sigma_{\chi, 1}(n)q^n,
\]
and
\[
E_{1, \chi}(z) \coloneqq \frac{L(-1,\chi)}{2} + \sum_{n \geq 1} \sigma_{1,\chi}(n)q^n,
\]
where $L(s, \chi)$ is the Dirichlet $L$-function of $\chi$.  (Although the constant term of the $q$-expansion at infinity of $E_{\chi,1}$ is $0$, it is not a cusp form as there exist cusps at which its $q$-expansion has a nonzero constant term.)  Both $E_{1, \chi}$ and $E_{\chi, 1}$ are normalized eigenforms with the following eigenvalues.
\[
\begin{array}{|c|c|c|}
\hline
 & T_\ell\text{-eigenvalue},\, \ell \neq N & T_N\text{-eigenvalue}\\
 \hline
E_{\chi, 1} & \sigma_{\chi, 1}(\ell) = \chi(\ell) + \ell & N\\
E_{1, \chi} & \sigma_{1, \chi}(\ell) = 1 + \chi(\ell)\ell & 1\\
\hline
\end{array}
\]
\begin{lemma}\label{lem:conglevelN}
The Eisenstein series $E_{\chi, 1}$ is congruent to $E_{2, N}$ modulo $\pp$ away from $N$ if and only if $\chi$ factors through $\Delta$.  The analogous statement is true for $E_{1,\chi}$.  
\end{lemma}

\begin{proof}
By definition of $\Delta$, a mod-$N$ Dirichlet character factors through $\Delta$ if and only if it has $p$-power order.  If $\chi$ has $p$-power order then it is immediate from the definitions of $E_{\chi,1}$ and $E_{1,\chi}$ that their Fourier coefficients satisfy the desired congruences. 
Conversely, suppose that $E_{2,N}$ and $E_{\chi,1}$ are congruent modulo $\pp$ away from $N$.  Then for all primes $\ell \neq N$ we have
\[
1 + \ell = a_\ell(E_{2,N}) \equiv a_\ell(E_{\chi,1}) = \bar{\chi}(\ell) + \ell \bmod \pp.
\]
By the Chebotarev density theorem, it follows that $\tr(1 \oplus \omega) = \tr(\bar{\chi} \oplus \omega)$, and thus $\{1,\omega\} = \{\bar{\chi},\omega\}$ by the Brauer--Nesbitt theorem (over $\overline{\F}_p$).  Thus $\bar{\chi} = 1$ and so $\chi$ has $p$-power order.  An analogous argument works for $E_{1,\chi}$ in place of $E_{\chi,1}$.
\end{proof}
\subsubsection{Level \texorpdfstring{$\Gamma_1(N^2)$}{}}\label{subsubsec:levelGamma1N2}
We begin by describing the Eisenstein series of level $\Gamma_1(N^2)$ that are eigenforms whose eigensystems do not appear in lower level.  Let $\chi_1,\chi_2$ be Dirichlet characters with conductors $c_1, c_2$, respectively.  Assume that $c_1c_2 = N^2$ and $(\chi_1\chi_2)(-1)=1$.  We have
\[
E_{\chi_1,\chi_2}(z) \coloneqq \begin{cases} 
\sum_{n \geq 1} \sigma_{\chi_1,\chi_2}(n)q^n & c_1 \neq 1\\
\frac{L(-1,\chi_2)}{2} + \sum_{n \geq 1} \sigma_{\chi_1,\chi_2}(n)q^n & c_1 = 1,
\end{cases}
\]
which is in $M_2(\Gamma_1(N^2))$.  Each $E_{\chi_1,\chi_2}$ is a normalized eigenform with the following eigenvalues.
\[
\begin{array}{|c|c|c|}
\hline
 & a_\ell(E_{\chi_1,\chi_2}),\, \ell \neq N & a_N(E_{\chi_1,\chi_2})\\
 \hline
c_1=c_2 = N & & 0\\
c_1=1 & \sigma_{\chi_1, \chi_2}(\ell) = \chi_1(\ell) + \chi_2(\ell)\ell  & 1\\
c_2=1 & & N\\
\hline
\end{array}
\]

\begin{remark}\label{rem:EisNeben}
The Nebentypus of $E_{\chi_1,\chi_2}$ is $\chi_1\chi_2$.  In particular, $E_{\chi_1, \chi_2} \in M_2(\Gamma_0(N^2))$ if and only if $\chi_2 = \chi_1^{-1}$.
\end{remark}

\begin{lemma}\label{lem:conglevelN2}
Let $\chi_i$ be Dirichlet characters with conductors $c_i$ such that $c_1c_2 = N^2$ and $(\chi_1\chi_2)(-1) = 1$.  The Eisenstein series $E_{\chi_1,\chi_2}$ is congruent to $E_{2,N}$ modulo $\pp$ away from $N$ if and only if $\chi_1$ and $\chi_2$ have $p$-power order.  In that case, we necessarily have $c_1=c_2=N$ and both $\chi_i$'s factor through~$\Delta$.
\end{lemma}

\begin{proof}
The final sentence follows from the previous part since all $p$-power order characters of $(\Z/N^2\Z)^\times$ factor through $(\Z/N\Z)^\times$.  Indeed, the kernel of the reduction map is an $N$-group. 

If $\chi_1$ and $\chi_2$ have $p$-power order, then it follows directly from the definition of $E_{\chi_1, \chi_2}$ that its $\ell$-th Fourier coefficient is congruent to $1+\ell = a_\ell(E_{2,N})$ modulo $\pp$ for all primes $\ell \neq N$.  Since $E_{\chi_1,\chi_2}$ is an eigenform, this suffices to show that $E_{\chi_1,\chi_2}$ is congruent to $E_{2,N}$ modulo $\pp$ away from $N$.  Conversely, suppose that $a_\ell(E_{2,N}) \equiv a_\ell(E_{\chi_1,\chi_2})$ for all primes $\ell \neq N$; that is, $1+\ell \equiv \chi_1(\ell) + \chi_2(\ell)\ell \bmod p$.  By the Chebotarev density theorem, it follows that $\tr(1\oplus\omega) = \tr(\bar{\chi}_1 \oplus \bar{\chi}_2\omega)$.  By the Brauer--Nesbitt theorem, we have $\{1, \omega\} = \{\bar{\chi}_1, \bar{\chi}_2\omega\}$.  Since both $\chi_i$'s are unramified at $p$ while $\omega$ is ramified at $p$, we must have $\bar{\chi}_1 = 1 = \bar{\chi}_2$.  Thus $\chi_1$ and $\chi_2$ have $p$-power order.
\end{proof}

Next we consider the Eisenstein series of level $\Gamma_1(N^2)$ that come from level $\Gamma_1(N)$, and hence their prime-to-$N$ eigensystems appear with multiplicity 2.  For any normalized eigenform $f \in M_2(\Gamma_1(N); \overline{\Z}_p)$, recall that both $f(z)$ and $f(Nz)$ define modular forms of level $\Gamma_1(N^2)$.  Let $M_f$ be the $\overline{\Z}_p$-span of $f$ and $f(Nz)$ in $M_2(\Gamma_1(N^2); \overline{\Z}_p)$, which is stable under all of the Hecke operators.  Every element of $M_f$ has $T_\ell$-eigenvalue equal to $a_\ell(f)$ for all primes $\ell \neq N$.

\begin{proposition}\label{prop:stabilizesoTN=0}
Fix a normalized eigenform $f \in M_2(\Gamma_1(N); \overline{\Z}_p)$.
\begin{enumerate}[label=(\roman*),leftmargin=*]
\item There is a normalized $f_0 \in M_f$ such that $T_Nf_0 = 0$.  Explicitly, $f_0(z) = f(z) - a_N(f)f(Nz)$. 
\item The form $f_0$ is the unique normalized $T_N$-eigenform in $M_f$ whose $T_N$-eigenvalue is congruent to $0$ modulo $\pp$.
\item If $1_N$ denotes the trivial mod-$N$ Dirichlet character, then $f_0 = f \otimes 1_N$.
\item The Nebentypus of $f_0$ is the same as that of $f$. 
\end{enumerate}

\end{proposition}

\begin{proof}
We know that $T_Nf = a_N(f)f$.  From the definition of $T_N$ it follows that $T_N$ sends $f(Nz)$ to $f$, so $f_0(z) = f(z) - a_N(f)f(Nz)$ satisfies $T_Nf_0 = 0$.  

To see the uniqueness claim it suffices to show that $a_N(f)$ is a $p$-adic unit, since in that case the two eigenvalues $a_N(f)$ and $0$ of $T_N$ on $M_f$ are distinct modulo $\pp$ and hence $T_N$ can be integrally diagonalized.  If $f$ is an Eisenstein series, then $a_N(f)$ is $1$ or $N$ (see \cref{level1N}), both of which are $p$-adic units.  If $f \in S_2(\Gamma_0(N))$ then $a_N(f) = \pm 1$ \cite[Theorem 4.6.17(2)]{Miyakebook}.  If $f \in S_2(\Gamma_1(N)) \setminus S_2(\Gamma_0(N))$ then $f$ is ramified principal series at $N$ \cite[Proposition 2.8(1)]{LoefflerWeinstein}.  Therefore its associated $p$-adic Galois representation $\rho_f$ is of the form
\[
\rho_f|_{G_{\Q_N}} \sim \begin{pmatrix}
\chi_1 & *\\
0 & \chi_2
\end{pmatrix}
\]
with $\chi_1$ the unramified character sending $\Frob_N$ to $a_N(f)$ and $\chi_1\chi_2 = \varepsilon\chi_f$.  Thus $a_N(f) = \chi_1(\Frob_N)$ is a $p$-adic unit.

Since $f \otimes 1_N = \sum_{n=0}^\infty 1_N(n)a_n(f)q^n \in M_2(\Gamma_1(N^2); \overline{\Z}_p)$ is a prime-to-$N$ normalized eigenform with $T_\ell$-eigenvalue $a_\ell(f)$ for all primes $\ell \neq N$, we see that $f \otimes 1_N \in M_f$.  Moreover $T_N(f \otimes 1_N) = 0$, so by uniqueness we have $f \otimes 1_N = f_0$.

The final claim follows from the fact that the Nebentypus character is determined by the action of the Hecke operators away from $N$, which act the same way on $f_0$ and $f$.
\end{proof}

\subsection{Hecke algebras}\label{subsec:Heckealgs}
In this section we define the relevant Hecke algebras.  There is a slight difference depending on whether the level is prime or the square of a prime; in the latter case we include the $T_N$ operator to cut out duplicate old eigensystems while in prime level we omit the $T_N$ operator both since it is not needed and for ease of comparison with the work of \cite{WWE}.  This distinction does not make much difference in practice since we immediately show in \cref{lem:reduced} that $T_N$ is $0$ in the prime-square level Hecke algebras of interest.

Let $\Gamma = \Gamma_i(N^j)$ for $i \in \{0,1\}$ and $j \in \{1,2\}$.  Let $H(\Gamma)$ be the $\Z_p$-subalgebra of $\End_{\Z_p}(M_2(\Gamma))$ that is generated by $\{T_\ell \colon \ell \text{ prime}\} \cup \{\langle d \rangle \colon d \in (\Z/N^j\Z)^\times\}$.  Let $H'(\Gamma)$ be the anemic subalgebra of $H(\Gamma)$; that is, $H'(\Gamma)$ is the $\Z_p$-subalgebra of $\End_{\Z_p}(M_2(\Gamma))$ generated by $\{T_\ell \colon \ell \nmid N \text{ prime}\} \cup \{\langle d \rangle \colon d \in (\Z/N^j\Z)^\times\}$.  A containment relation among different choices of $\Gamma$ induces a surjection of the corresponding Hecke algebras.

Since $E_{2,N} \in M_2(\Gamma_0(N))$ is an eigenform with $T_\ell$-eigenvalue $\ell+1$ for primes $\ell \neq N$, there is a $\Z_p$-algebra homomorphism $\lambda'_{E_{2,N}} =\lambda'_{E_{2,N}}(\Gamma_0(N)) \colon H'(\Gamma_0(N)) \to \Z_p$ that sends $T_\ell$ to $1+\ell$ for all primes $\ell \neq N$ and $\langle d \rangle$ to $1$ for all $d \in (\Z/N\Z)^\times$.  By precomposing $\lambda'_{E_{2,N}}$ with restriction maps, we obtain 
\[
\lambda'_{E_{2,N}}(\Gamma) \colon H'(\Gamma) \twoheadrightarrow \Z_p
\]
that sends $T_\ell$ to $\ell + 1$ for all primes $\ell \nmid N$ and $\langle d \rangle$ to $1$ for all $d \in (\Z/N^j\Z)^\times$.  Composing $\lambda'_{E_{2,N}}(\Gamma)$ with the reduction map $\Z_p \to \F_p$ gives a morphism $\bar{\lambda}'_{E_{2,N}}(\Gamma) \colon H'(\Gamma) \twoheadrightarrow \F_p$ whose kernel we call $\m'_\Gamma$ --- a maximal ideal of $H'(\Gamma)$.  Explicitly, 
\[
\m'_\Gamma = \langle p, T_\ell - \ell - 1, \langle d \rangle - 1 \colon \ell \nmid N \text{ prime}, d \in (\Z/N^j\Z)^\times \rangle.
\]

Since $T_N E_{2,N} = E_{2,N}$, there is a unique way to extend $\lambda'_{E_{2,N}}(\Gamma_0(N))$ to $H(\Gamma_0(N))$, namely by sending $T_N$ to $1$.  Call this map $\lambda_{E_{2,N}} \colon H(\Gamma_0(N)) \twoheadrightarrow \Z_p$.  Once again, precomposing $\lambda_{E_{2,N}}$ with restriction maps yields morphisms
\[
\lambda_{E_{2,N}}(\Gamma) \colon H(\Gamma) \twoheadrightarrow \Z_p
\]
that sends $T_N$ to $1$, $T_\ell$ to $\ell + 1$ for all primes $\ell \nmid N$ and $\langle d \rangle$ to $1$ for all $d \in (\Z/N^j\Z)^\times$.  Let $\bar{\lambda}_{E_{2,N}}(\Gamma)$ be its mod-$p$ reduction and $\m_{\Gamma, 1}$ the kernel of $\bar{\lambda}_{E_{2,N}}(\Gamma)$.  Explicitly,
\[
\m_{\Gamma,1} = \langle p, T_N - 1, T_\ell - \ell - 1, \langle d \rangle - 1 \colon \ell \nmid N \text{ prime}, d \in (\Z/N^j\Z)^\times \rangle.
\]

When $j = 2$ there is a second way to extend $\lambda'_{E_{2,N}}(\Gamma)$ to $H(\Gamma)$.  Namely, applying \cref{prop:stabilizesoTN=0} to $E_{2,N}$, we find that there is a unique normalized eigenform \[E_{1,1}(z) \coloneqq E_{2,N}(z) - a_N(E_{2,N})E_{2,N}(Nz) = E_{2,N}(z) - E_{2,N}(Nz)\] in $M_2(\Gamma_0(N^2))$ with $T_\ell$-eigenvalue $1+\ell$ for all primes $\ell \neq N$ and $T_N$-eigenvalue $0$.  Thus there is a $\Z_p$-algebra homomorphism $\lambda_{E_{1,1}} = \lambda_{E_{1,1}}(\Gamma_0(N^2)) \colon H(\Gamma_0(N^2)) \to \Z_p$ sending $T_\ell$ to $1+\ell$ for all primes $\ell \neq N$, $T_N$ to $0$, and $\langle d \rangle$ to $1$ for all $d \in (\Z/N^2\Z)^\times$.  Again, restriction maps give a corresponding map 
\[
\lambda_{E_{1,1}}(\Gamma_1(N^2)) \colon H(\Gamma_1(N^2)) \twoheadrightarrow \Z_p, 
\]
which sends $T_N$ to $0$, $T_\ell$ to $\ell + 1$ for all primes $\ell \nmid N$, and $\langle d \rangle$ to $1$ for all $d \in (\Z/N^2\Z)^\times$.  Let $\bar{\lambda}_{E_{1,1}}(\Gamma)$ be its mod-$p$ reduction and $\m_{\Gamma,0}$ the kernel of $\bar{\lambda}_{E_{1,1}}(\Gamma)$.  Explicitly,
\[
\m_{\Gamma,0} = \langle p, T_N, T_\ell - \ell - 1, \langle d \rangle - 1 \colon \ell \nmid N \text{ prime}, d \in (\Z/N^2\Z)^\times \rangle.
\]

\begin{definition}\label{defn:TGamma}
Let $\Gamma = \Gamma_i(N^j)$ for $i \in \{0,1\}$ and $j \in \{1,2\}$.
\begin{itemize}[leftmargin=*]
    \item If $j = 1$, let $\TT_\Gamma$ be the completion of $H'(\Gamma)$ at $\m'_{\Gamma}$, which is a local ring whose maximal ideal we denote by $\m_\Gamma$.  The ring $\TT_\Gamma$ parametrizes modular forms in $M_2(\Gamma; \overline{\Z}_p)$ that are congruent to $E_{2,N}$ modulo $\pp$ away from $N$.
    \item If $j = 2$, let $\TT_\Gamma$ be the completion of $H(\Gamma)$ at $\m_{\Gamma,0}$, which is a local ring whose maximal ideal we denote by $\m_\Gamma$.  The ring $\TT_\Gamma$ parametrizes modular forms in $M_2(\Gamma; \overline{\Z}_p)$ that are congruent to $E_{1,1}$ modulo $\pp$.  
\end{itemize}
\end{definition}
\begin{remark}\label{rem:Heckealgobservations}
We make a few observations about \cref{defn:TGamma}. 
\begin{itemize}[leftmargin=*]
    \item The ring $\TT_{\Gamma_0(N)}$ is equal to the ring $\TT$ defined in \cite[\S 3.1.3]{WWE}, which is also Mazur's Hecke algebra \cite[\S II.7]{Mazur}.
    \item By definition there are natural restriction maps $\TT_{\Gamma_1(N^j)} \twoheadrightarrow \TT_{\Gamma_0(N^j)}$ for any $j \in \{1,2\}$.
    \item There is also a $\Z_p$-algebra homomorphism $\TT_{\Gamma_0(N^2)} \twoheadrightarrow \TT_{\Gamma_0(N)}$ that sends $T_\ell$ to $T_\ell$ for all primes $\ell \nmid N$; see \cref{subsubsec:mapfromTtolevelN}.
\end{itemize}
For the rest of the paper $\TT_{\Gamma}$ denotes the completions at the Eisenstein ideals defined in \cref{defn:TGamma}. 
\end{remark}

\begin{lemma}\label{lem:reduced}
The ring $\TT_\Gamma$ is a reduced ring.  Moreover $T_N$ acts by $0$ on $M_2(\Gamma)_{\m_\Gamma}$ for $\Gamma = \Gamma_0(N^2)$ or $\Gamma_1(N^2)$.
\end{lemma}

\begin{proof}
Since $\TT_\Gamma$ has no $p$-torsion, $\TT_\Gamma$ embeds in $\TT_\Gamma \otimes \overline{\Q}_p$.  If the generators of $\TT_\Gamma$ are semisimple on $M_2(\Gamma)_{\m_\Gamma} \otimes \overline{\Q}_p$, then $\TT_\Gamma \otimes \overline{\Q}_p$ is a commutative semisimple ring and hence a product of fields, from which it follows that $\TT_\Gamma$ is reduced.  

It is well known that $T_\ell$ is semisimple for all primes $\ell \neq N$, and thus $\TT_\Gamma$ is reduced for $\Gamma = \Gamma_0(N)$ or $\Gamma_1(N)$.  For the other levels, it suffices to show that $T_N$ is semisimple on $M_2(\Gamma_1(N^2))_{\m_{\Gamma_1(N^2)}} \otimes \overline{\Q}_p$, which follows if $T_N$ acts by $0$ on this space.  A basis for this space consists of the union of:
\begin{itemize}[leftmargin=*]
\item primitive forms in $S_2(\Gamma_1(N^2))_{\m_{\Gamma_1(N^2)}} \otimes \overline{\Q}_p$,
\item $\{E_{\chi_1,\chi_2} \colon 1 \neq \chi_1,\chi_2 \colon \Delta \to \overline{\Q}_p^\times\}$, 
\item for each normalized eigenform $f \in M_2(\Gamma_1(N))_{\m_{\Gamma_1(N)}} \otimes \overline{\Q}_p$, the element $f_0$ from \cref{prop:stabilizesoTN=0}.
\end{itemize}
Note that $T_N$ acts by $0$ on the space of newforms in $S_2(\Gamma_1(N^2); \overline{\Q}_p)$ \cite[Theorem 4.6.17(3)]{Miyakebook}.  For nontrivial characters $\chi_i$ of $\Delta$ we have $a_N(E_{\chi_1,\chi_2}) = 0$ by definition, so $T_N$ acts by $0$ on those basis elements as well.  Finally, $T_Nf_0 = 0$ by \cref{prop:stabilizesoTN=0}, so in fact $T_N$ acts by $0$ on $M_2(\Gamma_1(N^2))_{\m_{\Gamma_1(N^2)}} \otimes \overline{\Q}_p$.
\end{proof}

Since $\TT_\Gamma$ is a reduced ring, there is an injection $\TT_\Gamma \hookrightarrow \prod_\pp \TT_{\Gamma}/\pp$, where $\pp$ runs over the minimal prime ideals of $\TT_\Gamma$.  The set of minimal prime ideals is in bijection with the set 
\[
\mathcal{F}_\Gamma \coloneqq \{f \in M_2(\Gamma)_{\m_\Gamma} \otimes \overline{\Q}_p \colon f \text{ normalized } \TT_\Gamma\text{-eigenform}\}.  
\]
In particular, if $f \in \mathcal{F}_\Gamma$ 
then we have an injection
\begin{align*}
\TT_\Gamma &\hookrightarrow \prod_{f \in \mathcal{F}_\Gamma} \overline{\Z}_p\\
T_\ell &\mapsto (a_\ell(f))_{f \in \mathcal{F}_\Gamma}.
\end{align*}
We often identify $\TT_\Gamma$ with its image under this map.

\subsubsection{The map \texorpdfstring{$\TT_{\Gamma_0(N^2)} \to \TT_{\Gamma_0(N)}$}{}}\label{subsubsec:mapfromTtolevelN}
There are natural restriction maps $\TT_{\Gamma_1(N^j)} \twoheadrightarrow \TT_{\Gamma_0(N^j)}$.  We also require a ring homomorphism $\TT_{\Gamma_0(N^2)} \to \TT_{\Gamma_0(N)}$.  This map is not immediate from the definitions.  Indeed, $M_2(\Gamma_0(N))_{\m_{\Gamma_0(N)}}$ is not a subset of $M_2(\Gamma_0(N^2))_{\m_{\Gamma_0(N^2)}}$ since $T_N$ acts by $+1$ on $M_2(\Gamma_0(N))_{\m_{\Gamma_0(N)}}$ while it acts by $0$ on $M_2(\Gamma_0(N^2))_{\m_{\Gamma_0(N^2)}}$, and thus the desired map is not simply given by restriction. Nevertheless, we have the following proposition. 


\begin{proposition}[{\cite[Corollary 2.12]{LMP}}]\label{prop:mapfromGamma0N2toGamma0N}
There is a $\Z_p$-algebra homomorphism $\TT_{\Gamma_0(N^2)} \twoheadrightarrow \TT_{\Gamma_0(N)}$ that sends $T_\ell$ to $T_\ell$ for all primes $\ell \nmid N$.
\end{proposition}

The key point in the proof of \cref{prop:mapfromGamma0N2toGamma0N} is establishing $\TT_{\Gamma_0(N^2)} \cong H'(\Gamma_0(N^2))_{\m'_{\Gamma_0(N^2)}}$; see \cite[Proposition 2.9]{LMP}. The map in \cref{prop:mapfromGamma0N2toGamma0N} is then obtained by composing this identification with the restriction map as follows:
\[\TT_{\Gamma_0(N^2)} \cong H'(\Gamma_0(N^2))_{\m'_{\Gamma_0(N^2)}} \twoheadrightarrow H'(\Gamma_0(N))_{\m'_{\Gamma_0(N^2)}} =\TT_{\Gamma_0(N)}.\]

\section{Deformation rings}\label{sec:DefRings}
In this section we introduce the pseudodeformation rings that will be matched with the Hecke algebras $\TT_\Gamma$ introduced in \cref{subsec:Heckealgs}.  We use the theory of pseudorepresentations due to Chenevier \cite{Chenevier}, which is summarized in \cite[\S 2]{LangWake25} in the 2-dimensional setting.  If $\rho$ is a representation, we denote the associated pseudorepresentation by $\psi(\rho)$, following \cite{WWE}.   Let $\overline{D} \colon G_{\Q,S} \to k$ be the pseudorepresentation $\psi(\omega \oplus 1)$.  All of our pseudorepresentations are continuous. 

Recall that to a 2-dimensional pseudorepresentation $D \colon G \to A$ on a group $G$, one can associate functions called the \textit{trace} and \textit{determinant} of $D$, denoted $\tr(D) \colon G \to A$ and $\det(D) \colon G \to A^\times$, respectively.  When $D = \psi(\rho)$ for a representation $\rho$, we have $\tr(D) = \tr(\rho)$ and $\det(D) = \det(\rho)$.  These functions satisfy the following axioms:
\begin{itemize}
\item $\det(D) \colon G \to A^\times$ is a group homomorphism;
\item $\tr(D)$ is central; that is, $\tr(D)(\sigma\tau) = \tr(D)(\tau\sigma)$ for all $\sigma, \tau \in G$;
\item $\tr(D)(1) = 2$;
\item $\tr(D)(\sigma\tau) + \det(D)(\tau)\tr(D)(\sigma\tau^{-1}) = (\tr(D)(\sigma))(\tr(D)(\tau))$ for all $\sigma, \tau \in G$.
\end{itemize}
Moreover, any pair of functions $(t \colon G \to A, d \colon G \to A^\times)$ satisfying the above axioms can be used to define a 2-dimensional $A$-valued pseudorepresentation on $G$ \cite[Example 1.8]{Chenevier}, \cite[\S 2]{LangWake25}.

A major advance in the theory of pseudodeformations came with the work of Wake and Wang-Erickson \cite{WWE2019}, where they put forward a general framework for defining deformation conditions for pseudorepresentations.  Their framework is especially useful for defining local conditions at $p$ that arise from $p$-adic Hodge theory.  A key insight of their work is that although a pseudorepresentation does not come with a module, they always arise from Cayley--Hamilton representations, which do have an associated module.  Via these Cayley--Hamilton representations, Wake and Wang-Erickson define what it means for a pseudodeformation to satisfy a deformation condition \cite[Definition 2.5.4]{WWE2019}.  We use their work to impose a finite flat condition at $p$ on the pseudodeformation rings we define in this section.


For a complete local noetherian $\Z_p$-algebra $A$, there are several deformation conditions we can impose on a pseudorepresentation $D \colon G_{\Q,S} \to A$ deforming $\overline{D}$:
\begin{enumerate}[label=(\roman*)]
\item\label{ff} (finite flat at $p$) $D$ is finite flat at $p$ in the sense of \cite[Definitions 2.5.4, 5.2.1]{WWE2019};
\item\label{fixdet} (fixed determinant) $\det(D) = \varepsilon$;
\item\label{uoS} (unramified-or-Steinberg at $N$) $D|_{I_N} = \psi(1 \oplus 1)$, so $D|_{I_N}$ is the trivial pseudorepresentation;
\item\label{sqfree} (squarefree level) $D|_{I_N} = \psi(\chi \oplus 1)$ for some character $\chi$ of $I_N$.
\end{enumerate}

Note that condition \ref{uoS} implies condition \ref{sqfree}.  We always impose condition \ref{ff} since our Hecke algebras are in weight 2 and have prime-to-$p$ level, so the associated Galois representations are necessarily finite flat at $p$.  

\begin{definition}
We define the following pseudodeformation rings.
\begin{itemize}[leftmargin=*] 
\item Let $R_{\Gamma_1(N^2)}$ be the universal pseudodeformation ring of $\overline{D}$ parametrizing pseudodeformations satisfying condition \ref{ff}, which exists by \cite[\S 5.2]{WWE2019}.  Write $D^{\univ} \colon G_{\Q,S} \to R_{\Gamma_1(N^2)}$ for the universal pseudodeformation; we also write $D_{\Gamma_1(N^2)} = D^{\univ}$.  All of the other deformation rings we define are quotients of $R_{\Gamma_1(N^2)}$.

\item Let $R_{\Gamma_0(N^2)}$ be the universal pseudodeformation ring of $\overline{D}$ parametrizing pseudodeformations satisfying conditions \ref{ff} and \ref{fixdet}.  It is the quotient of $R_{\Gamma_1(N^2)}$ by the ideal generated by $\{\det(D^{\univ})(\sigma) - \varepsilon(\sigma) \colon \sigma \in G_{\Q,S}\}$.   Write $D_{\Gamma_0(N^2)} \colon G_{\Q,S} \to R_{\Gamma_0(N^2)}$ for the associated universal pseudodeformation.

\item Let $R_{\Gamma_0(N)}$ be the universal pseudodeformation ring of $\overline{D}$ parametrizing pseudodeformations satisfying all (equivalently, the first three) of the above conditions.  The existence of $R_{\Gamma_0(N)}$ is proved in \cite[Proposition 4.1.1]{WWE}, where the ring is called $R$.  Moreover, they show that $R_{\Gamma_0(N)} \cong \TT_{\Gamma_0(N)}$ \cite[Corollary 7.1.3]{WWE}.

\item Define $R_{\Gamma_1(N)}$ as the quotient of $R_{\Gamma_1(N^2)}$ by the ideal  
\[
\mathfrak{a} \coloneqq \langle \tr(D^{\univ})(\sigma\tau) - 1 - (\tr(D^{\univ})(\sigma) - 1)(\tr(D^{\univ})(\tau) - 1) \colon \sigma, \tau \in I_N \rangle.
\]
Write $D_{\Gamma_1(N)} \colon G_{\Q,S} \to R_{\Gamma_1(N)}$ for the associated pseudodeformation.  The following lemma shows that $R_{\Gamma_1(N)}$ is the universal pseudodeformation ring of $\overline{D}$ parametrizing pseudodeformations satisfying conditions \ref{ff} and \ref{sqfree}.   
\end{itemize}
\end{definition}

\begin{lemma}
The ring $R_{\Gamma_1(N)}$ prorepresents the functor sending an Artinian local $\Z_p$-algebra $A$ with residue field $\F_p$ to the set of pseudodeformations $D \colon G_{\Q,S} \to A$ of $\overline{D}$ satisfying \ref{ff} and \ref{sqfree}. 
\end{lemma}

\begin{proof}
Let $A$ be an Artinian local $\Z_p$-algebra with residue field $\F_p$ and $D \colon G_{\Q, S} \to A$ a pseudodeformation of $\overline{D}$ satisfying conditions \ref{ff} and \ref{sqfree}.  The universality of $R_{\Gamma_1(N^2)}$ guarantees that $D$ arises from a morphism $\varphi \colon R_{\Gamma_1(N^2)} \to A$.  We only need to show that $\varphi$ factors through the quotient by $\mathfrak{a}$ if and only if $D$ satisfies \ref{sqfree}.
 
Define $\chi \colon I_N \to A$ by $\chi(\sigma) = \tr(D)(\sigma) - 1$, so $\mathfrak{a} = \langle \chi(\sigma\tau) - \chi(\sigma)\chi(\tau) \colon \sigma, \tau \in I_N\rangle$.  Thus $\varphi$ factors through the quotient by $\mathfrak{a}$ if and only if $\chi$ is a character.  If $\chi$ is a character, then $\tr(D)|_{I_N} = \chi + 1$.  Since 2-dimensional pseudorepresentations are determined by their trace in characteristic not 2 \cite[Example 7.6, Proposition 7.23]{Chenevier}, it follows that $D|_{I_N} = \psi(\chi \oplus 1)$.  Conversely, if there is a character $\chi' \colon I_N \to A^\times$ such that $D|_{I_N} = \psi(\chi' \oplus 1)$, then $\tr(D)|_{I_N} = \chi' + 1$, in which case $\chi = \chi'$ is a character. 
\end{proof}

\subsection{The map from \texorpdfstring{$R_\Gamma$}{} to \texorpdfstring{$\TT_\Gamma$}{}}\label{subsec:mapRtoT}
In this section we construct a surjective map $R_\Gamma \to \TT_\Gamma$ for each of the four choices of $\Gamma$.  When $\Gamma = \Gamma_0(N)$, the map $R_\Gamma \to \TT_\Gamma$ is constructed and shown to be an isomorphism in \cite[Proposition 4.2.4, Corollary 7.1.3]{WWE}, so we focus on the other three levels.  These maps are shown to be isomorphisms in \cref{sec:reductionstep} and \cref{sec:modrepthy}.   

\begin{proposition}\label{prop:RtoT}
There is a morphism of $\Z_p$-algebras $\Phi_\Gamma \colon R_\Gamma \to \TT_\Gamma$.
\end{proposition}

\begin{proof}
The strategy is to use the inclusion
\[
\TT_\Gamma \hookrightarrow \prod_{f \in \mathcal{F}_\Gamma} \overline{\Z}_p
\]
that sends $T_\ell$ to $(a_\ell(f))_{f \in \mathcal{F}_\Gamma}$ for all primes $\ell$.  In particular, we construct a morphism $R_\Gamma \to \prod_{f \in \mathcal{F}_\Gamma} \overline{\Z}_p$ and then show that its image lies in $\TT_\Gamma$.  This is also the approach taken in \cite[Lemma 3.1]{LangWake25} for level $\Gamma_0(N^2)$ when $N \equiv -1 \bmod p$.

Fix $f \in \mathcal{F}_\Gamma$.  If $f$ is a cuspform, then its associated Galois representation $\rho_f$ gives rise to a pseudorepresentation $D_f \coloneqq \psi(\rho_f) \colon G_{\Q,S} \to \overline{\Z}_p$.  Since $f \in M_2(\Gamma; \overline{\Z}_p)_{\m_\Gamma}$, it follows that $a_\ell(f) \equiv 1+\ell \bmod \pp$ for all primes $\ell \neq N$.  That is, $\psi(\rho_f) \equiv \overline{D} \bmod \pp$.  Moreover, $D_f$ is finite flat at $p$ since the level of $f$ is prime to $p$; indeed, $\rho_f$ comes from the $p$-adic Tate module of an abelian variety with good reduction at $p$.  Thus by universality there is a $\Z_p$-algebra homomorphism $\varphi_f \colon R_{\Gamma_1(N^2)} \to \overline{\Z}_p$ with the property that $\tr(D^{\univ})(\Frob_\ell)$ is sent to $a_\ell(f)$ for all primes $\ell \nmid Np$.  If $f \in \mathcal{F}_{\Gamma_0(N^2)}$, then $\det \rho_f = \varepsilon$ and hence $\varphi_f$ factors through the quotient to $R_{\Gamma_0(N^2)}$.  If $f \in \mathcal{F}_{\Gamma_1(N)}$ then $f$ is either Steinberg or ramified principal series at $N$.  In either case, $\rho_f|_{G_{\Q_N}}$ has an unramified submodule and hence $D_f$ satisfies condition \ref{sqfree}.  Thus in this case $\varphi_f$ factors through the quotient to $R_{\Gamma_1(N)}$.

Next assume that $f \in \mathcal{F}_\Gamma$ is an Eisenstein series.  Since $f$ is congruent to $E_{2,N}$ modulo $\pp$ away from $N$, it follows from \cref{lem:conglevelN}, \cref{lem:conglevelN2}, and \cref{prop:stabilizesoTN=0} that there are mod-$N$ Dirichlet characters $\chi_1, \chi_2$ that factor through $\Delta$ such that $a_\ell(f) = \chi_1(\ell) + \chi_2(\ell)\ell$ for all primes $\ell \neq N$.  Viewing $\chi_1, \chi_2$ as characters of $G_{\Q,S}$, we see that $\psi(\chi_1 \oplus \chi_2\varepsilon) \colon G_{\Q, S} \to \overline{\Z}_p$ is a pseudodeformation of $\overline{D}$ satisfying $\tr(\psi(\chi_1 \oplus \chi_2\varepsilon))(\Frob_\ell) = a_\ell(f)$ for all primes $\ell \nmid Np$.  Recall that a character $\chi$ of $G_{\Q, S}$ is finite flat at $p$ if and only if either $\chi$ is unramified at $p$ or $\chi = \chi'\varepsilon$ with $\chi'$ unramified at $p$.  In particular, both $\chi_1$ and $\chi_2\varepsilon$ are finite flat at $p$, from which it follows that $\psi(\chi_1\oplus\chi_2\varepsilon)$ is finite flat at $p$.  By universality there is a $\Z_p$-algebra homomorphism $\varphi_f \colon R_{\Gamma_1(N^2)} \to \overline{\Z}_p$ that sends $\tr(D^{\univ})(\Frob_\ell)$ to $a_\ell(f)$ for all primes $\ell \nmid Np$.  If $f \in \mathcal{F}_{\Gamma_0(N^2)}$ then we must have $\chi_2 = \chi_1^{-1}$ by \cref{rem:EisNeben}, in which case $\det(\chi_1\oplus\chi_2\varepsilon) = \varepsilon$.  Hence if $f \in \mathcal{F}_{\Gamma_0(N^2)}$ then $\varphi_f$ factors through the quotient to $R_{\Gamma_0(N^2)}$.  Finally, if $f \in \mathcal{F}_{\Gamma_1(N)}$ then at least one of $\chi_1,\chi_2$ must be trivial.  It follows that $(\chi_1\oplus\chi_2\varepsilon)|_{I_N}$ has a trivial submodule and hence $\psi(\chi_1\oplus\chi_2\varepsilon)$ satisfies \ref{sqfree}.  Thus if $f \in \mathcal{F}_{\Gamma_1(N)}$ then $\varphi_f$ factors through the quotient to $R_{\Gamma_1(N)}$.

Let $\Phi_\Gamma \colon R_\Gamma \to \prod_{f \in \mathcal{F}_\Gamma} \overline{\Z}_p$ be the map given by the product of all of the $\varphi_f$'s described in the last two paragraphs.  We show that the image of $\Phi_\Gamma$ lands in $\TT_\Gamma$.  By the Chebotarev density theorem, $R_\Gamma$ is topologically generated by the set of $\tr(D_\Gamma)(\Frob_\ell)$ for primes $\ell \nmid Np$.  For any such $\ell$ we have that 
\[
\Phi_\Gamma(\tr(D_\Gamma)(\Frob_\ell)) = (a_\ell(f))_{f \in \mathcal{F}_\Gamma} = T_\ell \in \TT_\Gamma,
\]
and hence the image of $\Phi_\Gamma$ is contained in $\TT_\Gamma$.
\end{proof}

Composing $D_\Gamma$ with $\Phi_\Gamma$ we obtain a $\TT_\Gamma$-valued pseudodeformation of $\overline{D}$, which we denote by $D_{\TT_\Gamma} \colon G_{\Q, S} \to \TT_\Gamma$.  It has the property that for any prime $\ell \nmid Np$, $\tr(D_{\TT_\Gamma})(\Frob_\ell) = T_\ell$ and $\det(D_{\TT_\Gamma})(\Frob_\ell) = \ell\langle \ell \rangle$.

\begin{proposition}\label{prop:RtoTsurj}
The maps $\Phi_\Gamma$ of \cref{prop:RtoT} are surjective onto $\TT_{\Gamma}$.
\end{proposition}

\begin{proof}
Since $D_{\TT_\Gamma} = \Phi_\Gamma \circ D_{\Gamma}$, we see that for primes $\ell \nmid Np$ we have
\[
T_\ell = \Phi_\Gamma(\tr(D_\Gamma)(\Frob_\ell)) \in \im \Phi_\Gamma \text{ and } \langle \ell \rangle = \Phi_\Gamma(\ell^{-1}\det(D_\Gamma)(\Frob_\ell)) \in \im \Phi_\Gamma.
\]
For $\Gamma \in \{\Gamma_0(N^2), \Gamma_1(N^2)\}$ we have $T_N = 0 \in \TT_\Gamma$ by \cref{lem:reduced}.  It remains to show that $T_p \in \im \Phi_\Gamma$.  The diagram
\[
\xymatrix{
R_{\Gamma_1(N^2)} \ar@{->>}[d] \ar@{->}[r]^{\Phi_{\Gamma_1(N^2)}} & \TT_{\Gamma_1(N^2)}  \ar@{->>}[d] \\
R_{\Gamma_0(N^2)} \ar@{->}[r]_{\Phi_{\Gamma_0(N^2)}} & \TT_{\Gamma_0(N^2)}
}
\]
commutes, and hence the surjectivity of $\Phi_{\Gamma_0(N^2)}$ follows from that of $\Phi_{\Gamma_1(N^2)}$.

Let $\Gamma = \Gamma_1(N)$ or $\Gamma_1(N^2)$.  That $T_p$ is in the image of $\Phi_\Gamma$ follows from the fact that the eigenforms giving $D_{\TT_\Gamma}$ are ordinary at $p$ as in \cite[Proposition 4.2.4]{WWE} and \cite[Lemma 3.1]{LangWake25}.  We give the details for the convenience of the reader.

For each $f \in \mathcal{F}_\Gamma$, we have $a_p(f) \equiv 1 + p \equiv 1 \bmod \pp$, which shows that $T_p \in \TT_\Gamma^\times$.  Thus the polynomial
\[
X^2 - T_pX + p\langle p \rangle \in \TT_\Gamma[X]
\]
has a unique unit root, which we denote by $U_p \in \TT_\Gamma^\times$.  Note that $T_p = U_p + U_p^{-1}p\langle p \rangle$.  Since $\langle p \rangle = \langle \ell \rangle$ for a prime $\ell \nmid Np$ and $\langle \ell \rangle \in \im \Phi_\Gamma$, it suffices to show that $U_p$ and $U_p^{-1}$ are in the image of $\Phi_\Gamma$.  Moreover $\im \Phi_\Gamma$ is a local ring, being the homomorphic image of a local ring.  By definition, if $U_p^{-1} \in \im \Phi_\Gamma$ then it is a unit modulo the maximal ideal of $\im \Phi_\Gamma$ and hence is a unit in $\im \Phi_\Gamma$.  Therefore $U_p \in \im \Phi_\Gamma$ as well, so it suffices to show that $U_p^{-1} \in \im \Phi_\Gamma$.  For each $f \in \mathcal{F}_\Gamma$ with Nebentypus $\chi$, let $\alpha_f$ be the unique unit root of $X^2 - a_p(f)X + p\chi(p)$, so $U_p$ corresponds to the tuple $(\alpha_f)_{f \in \mathcal{F}_\Gamma}$.

Fix a choice of $\Frob_p \in G_{\Q_p}$ and $\tau \in I_p$ such that $\omega(\tau) \neq 1 \in \F_p$.  We show that for each $f \in \mathcal{F}_\Gamma$ with associated pseudodeformation $D_f$ of $\overline{D}$, we have
\begin{equation}\label{apformula}
\alpha_f^{-1} = \frac{\tr(D_f)(\tau\Frob_p) - \tr(D_f)(\Frob_p)}{(\det(D_f)(\Frob_p))(\varepsilon(\tau)-1)}, 
\end{equation}
which is visibly in the image of the map $R_\Gamma \to \overline{\Z}_p$ coming from $D_f$.  It follows that 
\[
U_p^{-1} = \Phi_\Gamma\left(\frac{\tr(D_\Gamma)(\tau\Frob_p) - \tr(D_\Gamma)(\Frob_p)}{(\det(D_\Gamma)(\Frob_p)) (\varepsilon(\tau)-1)}\right) \in \im \Phi_\Gamma.
\]

Let $f \in \mathcal{F}_\Gamma$ with Nebentypus $\chi$.  We claim that  
\begin{equation}\label{eq:traceonGp}
\tr\rho_f(\sigma) = \lambda(\alpha_f)^{-1}(\sigma)\det \rho_f(\sigma) + \lambda(\alpha_f)(\sigma),
\end{equation}
where for any $x$, $\lambda(x)$ is the unramified character of $G_{\Q_p}$ that sends $\Frob_p$ to $x$.  If $f$ is a cusp form, it is ordinary at $p$ and \eqref{eq:traceonGp} follows from \cite[p. 250]{MazurWiles86}.  If $f$ is an Eisenstein series such that $D_f = \psi(\chi_1 \oplus \chi_2\varepsilon)$ for characters $\chi_i$ of $\Delta$, then \eqref{eq:traceonGp} can be checked directly.  Indeed, in this case $\alpha_f = \chi_1(p)$ and so $\lambda(\alpha_f) = \chi_1|_{G_{\Q_p}}$.  Applying \eqref{eq:traceonGp} to $\sigma = \tau\Frob_p$ and $\sigma = \Frob_p$ and subtracting them yields \eqref{apformula}.
\end{proof}

\begin{remark}
\cref{prop:RtoTsurj} shows that $\TT_\Gamma$ is generated by the Hecke operators away from $p$.  Indeed, by the Chebotarev density theorem and universality, $R_\Gamma$ is topologically generated by $\{\tr(D_\Gamma)(\Frob_\ell) \colon \ell \not\in S\}$.  Thus the image of $\Phi_\Gamma$ is topologically generated by the image of this set, namely $\{T_\ell \colon \ell \nmid Np\}$.  The surjectivity of $\Phi_\Gamma$ gives the result.
\end{remark}

\section{Bootstrapping from level \texorpdfstring{$\Gamma_0(N)$}{}}\label{sec:reductionstep}

 Across this section and \cref{sec:modrepthy}, we show how Wake--Wang-Erickson's result that $\Phi_{\Gamma_0(N)}$ is an isomorphism \cite[Corollary 7.1.3]{WWE} can be leveraged to show that $\Phi_{\Gamma}$ is an isomorphism for the other three choices of $\Gamma$. In \cref{sec:commalglemmas}, we present a pair of abstract commutative algebra lemmas that are used in this deduction.  Precisely how the results of \cref{sec:modrepthy} and the rest of \cref{sec:reductionstep} are used to deduce the main results stated in \cref{sec:introduction} from these abstract  commutative algebra lemmas is explained at the start of \cref{sec:modrepthy}. In \cref{subsec:countingarg}, we count $\Z_p$-ranks of $\TT_{\Gamma_0(N^2)}$, $\TT_{\Gamma_1(N)}$ and $\TT_{\Gamma_1(N^2)}$. Our strategy is to twist by appropriate $p$-power Dirichlet characters to show that counting ranks at levels $\Gamma_0(N^2)$ and $\Gamma_1(N)$ are equivalent, and either of the rank counts at these levels can be used to deduce the rank count at level $\Gamma_1(N^2)$. In \cref{subsec:redstep}, we equip  $R_{\Gamma_1(N)}$ and $R_{\Gamma_1(N^2)}$ with an action of $\Z_p[\Delta]$; this action is given by the diamond operators after pushing forward to the relevant Hecke algebra. For $\Gamma_1(M)$ with $M=N$ or $N^2$,  we also show that, modulo the augmentation ideal of $\Z_p[\Delta]$, the map $\Phi_{\Gamma_1(M)}$ reduces to the map $\Phi_{\Gamma_0(M)}$. 
 

\subsection{Commutative algebra lemmas}\label{sec:commalglemmas}
In this section we establish two abstract lemmas in commutative algebra, both easy consequences of Nakayama's Lemma, that are used repeatedly in reducing from any of the four levels $\Gamma$ to $\Gamma_0(N)$.  

\begin{lemma}\label{lem:NAKReduction}
Let $A$ be a local $\Z_p$-algebra with maximal ideal $\m$ that is $\m$-adically compact.  Fix an augmentation $A \to \Z_p$, and let $\II$ denote its kernel.  Let $\varphi \colon R \to T$ be a surjective morphism of local compact $A$-algebras.  Assume that $T$ is free as an $A$-module and that $\varphi$ induces an isomorphism $R/\II R \cong T/\II T$.  Then $\varphi$ is an isomorphism.
\end{lemma}

\begin{proof}
Let $I = \ker \varphi$, which is a closed ideal of $R$ and hence compact.  Since $T$ is $A$-free, it follows that $\varphi$ admits a splitting as $A$-modules; that is, we have an isomorphism of $A$-modules $R \cong T \oplus I$.  Tensoring with $A/\II$ gives $R/\II R \cong (T/\II T) \oplus (I/\II I)$.  Since $\varphi$ induces an isomorphism $R/\II R \cong T/\II T$, it follows that $I/\II I = 0$.  That is, $\II I = I$, so $\m I = I$.  By topological Nakayama's lemma, $I = 0$ and hence $\varphi$ is an isomorphism \cite[Exercise V.2.6]{Neukirch}.
\end{proof}

We now give a method for establishing the hypothesis of \cref{lem:NAKReduction} that $T$ is a free $A$-module.  It reduces the problem to showing that $T/\II T$ is $\Z_p$-free and counting ranks, the latter of which can be done after base change to $\overline{\Q}_p$.

\begin{lemma}\label{lem:NAKpluscounting}
Let $A$ be a finite free local $\Z_p$-algebra.  Fix an augmentation $A \to \Z_p$ with kernel $\II$.  Let $T$ be a finite local $A$-algebra such that $T$ and $T/\II T$ are free $\Z_p$-modules.  Then $T$ is free as an $A$-module if and only if $\rk_{\Z_p} T = (\rk_{\Z_p} A)(\rk_{\Z_p} T/\II T)$.
\end{lemma}

\begin{proof}
    The fact that $\rk_{\Z_p} T$ satisfies the given formula when $T$ is $A$-free is immediate.  Conversely, suppose that $\rk_{\Z_p} T = (\rk_{\Z_p} A)(\rk_{\Z_p} T/\II T)$ and set $r = \rk_{\Z_p} T/\II T$.  By Nakayama's Lemma, there is a surjection of $\Z_p[\Delta]$-modules $A^r \to T$.  Both $A^r$ and $T$ are free $\Z_p$-modules of rank $r \cdot (\rk_{\Z_p} A)$, so the map is an isomorphism.
\end{proof}

\subsection{Counting arguments}\label{subsec:countingarg}
In this section we relate the $\Z_p$-ranks of $\TT_{\Gamma_1(N^2)}$ and $\TT_{\Gamma_1(N)}$ to that of $\TT_{\Gamma_0(N^2)}$.  These calculations allow us to use \cref{lem:NAKpluscounting} in the next section to reduce to studying only level $\Gamma_0(N^2)$.  Moreover, calculating ranks can be done over $\overline{\Q}_p$ by counting eigenforms.

The key tool we use is twisting modular forms by Dirichlet characters.  For a modular form $f$ and Dirichlet character $\chi$, one defines the twist of $f$ by $\chi$ as $f \otimes \chi \coloneqq \sum_{n=0}^\infty \chi(n)a_n(f)q^n$.  Its Nebentypus is $\chi^2$ times the Nebentypus of $f$, and its level is the least common multiple of the level of $f$, the square of the conductor of $\chi$, and the product of the conductor of $\chi$ with the conductor of the Nebentypus of $f$ \cite[Proposition 3.64]{Shimura}.  In particular, if $f$ has level $N^2$ with Nebentypus of conductor dividing $N$, then $f \otimes \chi$ has level $N^2$ for any character $\chi$ of conductor dividing $N$.  
Twisting preserves the property of being normalized, and it is easy to check that if $f$ is a Hecke eigenform away from its level, then so is $f \otimes \chi$.  Thus we can define a function 
\begin{align*}
    \alpha \colon \mathcal{F}_{\Gamma_0(N^2)} \times \hat{\Delta} &\to \mathcal{F}_{\Gamma_1(N^2)}\\
    (f, \chi) &\mapsto f \otimes \chi,
\end{align*}
where $\hat{\Delta}$ denotes the group of characters on $\Delta$.  This is well defined since tensoring by a $p$-power order character preserves the property of being Eisenstein modulo a prime above $p$.  

To define an inverse map to $\alpha$, recall that any $g \in \mathcal{F}_{\Gamma_1(N^2)}$ has Nebentypus $\chi_g$ congruent to 1 modulo a prime above $p$ and hence is a character of $\Delta$.  Since $\hat{\Delta}$ is a $p$-group and $p \neq 2$, for any $\chi \in \hat{\Delta}$ there is a unique element in $\hat{\Delta}$ that squares to $\chi$, which we denote by $\chi^{1/2}$.  Define 
\begin{align*}
\beta \colon \mathcal{F}_{\Gamma_1(N^2)} &\to \mathcal{F}_{\Gamma_0(N^2)} \times \hat{\Delta}\\
g &\mapsto (g \otimes \chi_g^{-1/2}, \chi_g^{1/2}).
\end{align*}
The same argument as above shows that $g \otimes \chi_g^{-1/2}$ is a normalized eigenform for $\TT_{\Gamma_1(N^2)}$.  By construction $g \otimes \chi_g^{-1/2}$ has trivial Nebentypus, so $g \otimes \chi_g^{-1/2} \in \mathcal{F}_{\Gamma_0(N^2)}$.

\begin{proposition}\label{prop:countingranksN2}
The $\Z_p$-rank of $\TT_{\Gamma_1(N^2)}$ is equal to $(\rk_{\Z_p} \TT_{\Gamma_0(N^2)})|\Delta|=(\rk_{\Z_p} \TT_{\Gamma_0(N^2)})p^s$.
\end{proposition}

\begin{proof}
    Since $\TT_{\Gamma}$ is $\Z_p$-free, we can compute ranks after tensoring to $\overline{\Q}_p$, over which ranks count Hecke eigensystems.  
    Thus it suffices to show that there is a bijection between $\mathcal{F}_{\Gamma_0(N^2)} \times \hat{\Delta}$ and $\mathcal{F}_{\Gamma_1(N^2)}$.
It is immediate to check that the maps $\alpha$ and $\beta$ defined above are inverse maps, which proves the proposition.
\end{proof}

The functions $\alpha$ and $\beta$ do not restrict to bijections between $\mathcal{F}_{\Gamma_0(N)} \times \hat{\Delta}$ and $\mathcal{F}_{\Gamma_1(N)}$.  However, we will see that it is possible to relate $\mathcal{F}_{\Gamma_1(N)}$ with $\mathcal{F}_{\Gamma_0(N^2)}$ via twisting.  There is a natural injection $\iota \colon \mathcal{F}_{\Gamma_1(N)} \to \mathcal{F}_{\Gamma_1(N^2)}$ given by $f \mapsto f_0 = f \otimes 1_N$ given by \cref{prop:stabilizesoTN=0}.  We suppress $\iota$ and view $\beta \colon \mathcal{F}_{\Gamma_1(N)} \to \mathcal{F}_{\Gamma_0(N^2)} \times \hat{\Delta}$ as an injection given by $g \mapsto (g \otimes \chi_{g}^{-1/2}, \chi_g^{1/2})$.  The strategy is to compute $\# \mathcal{F}_{\Gamma_1(N)}$ by computing the image of $\beta$.  

To do so, it is useful to know when $f \otimes \chi$ is a newform of level $N^2$ and when it arises in level $N$.  Certainly $\alpha(\mathcal{F}_{\Gamma_0(N)} \times \{1\}) = \mathcal{F}_{\Gamma_0(N)}$, so assume that $f \in \mathcal{F}_{\Gamma_0(N^2)}$ is a cusp form that is new of level $N^2$.  For which values of $\chi$ is there a form $g \in \mathcal{F}_{\Gamma_1(N)}$ such that $g_0 = f \otimes \chi$?  Recall that such an $f$ is either ramified principal series or supercuspidal at $N$.  The next lemma shows that the latter case does not occur when $f$ is congruent to $E_{2,N}$ away from $N$.

\begin{lemma}\label{lem:EisnotSteinberg}
Let $f \in S_2(\Gamma_1(N^2))$ be a primitive form that is congruent to $E_{2, N}$ modulo $\pp$ away from $N$.  Then $f$ is not supercuspidal at $N$, so $f$ is ramified principal series at $N$.
\end{lemma}

\begin{proof}
Suppose that $f$ is supercuspidal at $N$.  By the local Langlands correspondence, $\rho_f|_{G_{\Q_N}} \cong \Ind_K^{\Q_N} \eta$ for a quadratic extension $K/\Q_N$ and character $\eta$ on $G_K$ such that $\eta^\sigma \neq \eta$, where $1 \neq \sigma \in \Gal{K}{\Q_N}$ and $\eta^\sigma(g) \coloneqq \eta(\sigma g \sigma^{-1})$.  As $f$ is congruent to $E_{2,N}$ modulo $\pp$ away from $N$, $\rho_f|_{G_{\Q_N}}$ is reducible modulo $\pp$; that is, $\eta \equiv \eta^\sigma \bmod \pp$.

It follows that $\bar{\eta} := \eta \bmod \pp$ admits an extension $\bar{\eta}'$ to $G_{\Q_N}$, and the only other extension is $\chi_{K}\bar{\eta}'$, where $\chi_K$ is the quadratic character associated with $K/\Q_N$.  Then we must have that
\[
\Ind_K^{\Q_N} \bar{\eta} = \bar{\eta}' \oplus  \chi_K\bar{\eta}'.
\]
In particular, since $\chi_K$ is nontrivial, it follows that $\tr \rho_f|_{G_{\Q_N}} = \tr \Ind_K^{\Q_N} \eta \equiv \bar{\eta}' + \chi_K\bar{\eta}' \bmod \pp$ is not identically equal to $2 = \tr(1 \oplus \omega)|_{G_{\Q_N}}$. This contradicts the assumption that $f$ is congruent to $E_{2,N}$ modulo $\pp$ away from $N$. 
\end{proof}

\cref{lem:EisnotSteinberg} ensures that any cusp form $f \in \mathcal{F}_{\Gamma_0(N^2)}$ that is new at level $N^2$ is ramified principal series at $N$, so
\[
\rho_f|_{I_N} \sim \begin{pmatrix}
\eta & 0\\
0 & \eta^{-1}
\end{pmatrix}
\]
for some $\eta \in \hat{\Delta}$.  For any $\chi \in \hat{\Delta}$, write $(f \otimes \chi)^{\prim}$ for the primitive form of level dividing $N^2$ that has the same eigenvalues as $f \otimes \chi$ away from $N$.  We have
\[
\rho_{(f \otimes \chi)^{\prim}}|_{I_N} \sim \begin{pmatrix}
\eta\chi & 0\\
0 & \eta^{-1}\chi
\end{pmatrix}.
\]
The key fact is that $(f \otimes \chi)^{\prim}$ has level $N$ if and only if either $\eta\chi = 1$ or $\eta^{-1}\chi = 1$ \cite[\S 2.3]{LoefflerWeinstein}.  


\begin{proposition}\label{prop:pisurjective}
Every $f \in \mathcal{F}_{\Gamma_0(N^2)}$ is of the form $g \otimes \chi_g^{-1/2}$ for some $g \in \mathcal{F}_{\Gamma_1(N)}$.  If $f \in \mathcal{F}_{\Gamma_0(N)}$ then there is a unique such $g$; otherwise, there are exactly two such $g$'s.
\end{proposition}

\begin{proof}
Let $f \in \mathcal{F}_{\Gamma_0(N^2)}$.  If $f$ is a cusp form that is new at level $N^2$ then \cref{lem:EisnotSteinberg} ensures that $f$ is ramified principal series at $N$.  Let $\eta \in \hat{\Delta}$ be a character such that $\tr \rho_f|_{I_N} = \eta + \eta^{-1}$; note that $\eta \neq 1$ since $\rho_f$ is ramified at $N$.  By the discussion preceding this proposition, we see that $(f \otimes \chi)^{\prim}$ has level $N$ if and only if $\chi \in \{\eta, \eta^{-1}\}$.  Thus $(f \otimes \eta)^{\prim}$ and $(f \otimes \eta^{-1})^{\prim}$ are the unique forms $g$ satisfying $g \otimes \chi_g^{-1/2} = f$.  

Suppose that $f$ is a cusp form that is not new at level $N^2$.  By \cref{prop:stabilizesoTN=0}, $f = g \otimes 1_N$ for some $g \in \mathcal{F}_{\Gamma_0(N)}$.  To see that $g$ is unique, note that $g$ is the unique element $h \in \mathcal{F}_{\Gamma_0(N)}$ such that $h_0 = f$.  Any $h \in \mathcal{F}_{\Gamma_1(N)} \setminus \mathcal{F}_{\Gamma_0(N)}$ is ramified principal series at $N$, hence so is any twist of $h$.  If there were $\chi \in \hat{\Delta}$ such that $h \otimes \chi = f = g_0 = g \otimes 1$, then $g$ would be ramified principal series.  But $g$ is Steinberg at $N$, a contradiction; so $g$ is unique.

If $f \in \mathcal{F}_{\Gamma_0(N^2)}$ is an Eisenstein series, then it is easy to check the claim by hand.  Namely $E_{1,1} = E_{2,N} \otimes 1_N$ and for $1 \neq \chi \in \hat{\Delta}$, the forms $E_{1,\chi}, E_{\chi^{-1},1} \in \mathcal{F}_{\Gamma_1(N)}$ are the two forms $g$ such that $g \otimes \chi_g^{-1/2} = E_{\chi^{-1/2}, \chi^{1/2}}$.
\end{proof}

\begin{corollary}\label{cor:countingranksequivalence}
We have $\rk_{\Z_p} \TT_{\Gamma_1(N)} = 2\rk_{\Z_p} \TT_{\Gamma_0(N^2)} - \rk_{\Z_p} \TT_{\Gamma_0(N)}$.  Thus the following are equivalent:
\begin{enumerate}[label=(\roman*),leftmargin=*]
\item $\rk_{\Z_p} \TT_{\Gamma_0(N^2)} = \bigl(\frac{|\Delta| + 1}{2}\bigr)\rk_{\Z_p} \TT_{\Gamma_0(N)}=\bigl(\frac{p^s + 1}{2}\bigr)r$;
\item $\rk_{\Z_p} \TT_{\Gamma_1(N)} = |\Delta|\rk_{\Z_p} \TT_{\Gamma_0(N)} = p^sr$.
\end{enumerate}
\end{corollary}

\begin{proof}
As in the proof of \cref{prop:countingranksN2}, we can compute ranks after tensoring to $\overline{\Q}_p$ and simply count Hecke eigensystems.  It suffices to show that the size of the image of the injective map $\beta \colon \mathcal{F}_{\Gamma_1(N)} \to \mathcal{F}_{\Gamma_0(N^2)} \times \hat{\Delta}$ has size $2\#\mathcal{F}_{\Gamma_0(N^2)} - \#\mathcal{F}_{\Gamma_0(N)}$, which follows from \cref{prop:pisurjective}.  The equivalence follows directly from the first equation.
\end{proof}

\subsection{The reduction step}\label{subsec:redstep}
In this section, let $M = N$ or $N^2$.  The $\Z_p$-algebra generated by the diamond operators of $p$-power order is isomorphic to $\Z_p[\Delta]$.  We wish to apply \cref{lem:NAKReduction} with $A = \Z_p[\Delta]$, which carries a natural augmentation $\Z_p[\Delta] \twoheadrightarrow \Z_p$ with kernel $\II$.  The following standard argument shows that $R_{\Gamma_1(M)}$ is a $\Z_p[\Delta]$-module.  

\begin{lemma}\label{lem:ZpDeltAlg}
The ring $R_{\Gamma_1(M)}$ is naturally a $\Z_p[\Delta]$-algebra, and $R_{\Gamma_1(M)}/\II R_{\Gamma_1(M)} \cong R_{\Gamma_0(M)}$.
\end{lemma}

\begin{proof}
We show that $\Z_p[\Delta]$ is the universal finite flat-at-$p$ deformation ring of $\omega = \det(\overline{D})$.  Since $\det(D_{\Gamma_1(M)})$ is a finite flat-at-$p$ deformation of $\omega$, by universality we obtain the desired ring homomorphism $\Z_p[\Delta] \to R_{\Gamma_1(M)}$.

The universal deformation ring of a character $\chi \colon G \to \F_p^\times$ is given by $\Z_p\lb G^{\ab,\mathrm{pro-}p} \rb$, and the universal deformation is the Teichmüller lift of $\chi$ times the tautological character.  In our case, by class field theory we compute that 
\[
G_{\Q, S}^{\ab, \mathrm{pro-}p} = \Gal{\Q(\mu_{p^\infty}, \mu_{N^\infty})}{\Q}^{\mathrm{pro-}p} \cong (\Z_p^\times \times \Z_N^\times)^{\mathrm{pro-}p} = (1 + p\Z_p) \times \Delta.
\]
Recall that the only finite flat deformations of $\omega$ are unramified characters times $\varepsilon$ so finite flat characters are forced to be the pro-$p$ part of $\varepsilon$ on $1+p\Z_p$.  Therefore the quotient of $\Z_p\lb (1+p\Z_p) \times \Delta \rb$ that parametrizes finite flat-at-$p$ deformations of $\omega$ is $\Z_p[\Delta]$, as claimed.

The last claim follows from the fact that the quotient by $\II$ exactly forces the universal pseudorepresentation to have determinant equal to $\varepsilon$, which is the defining condition of $R_{\Gamma_0(M)}$.
\end{proof}

\begin{remark}\label{rem:diamondop}
The image of $\Z_p[\Delta]$ under $\Phi_\Gamma$ is the $\Z_p$-subalgebra of $\TT_\Gamma$ generated by $\{\langle d \rangle \colon d \in \Delta\}$.
\end{remark}

We now show that $\Phi_{\Gamma_1(M)}$ modulo $\II$ is an isomorphism, which is needed to apply \cref{lem:NAKReduction}.  We do this assuming that $\Phi_{\Gamma_0(M)}$ is an isomorphism, which is known when $M = N$ \cite[Corollary 7.1.3]{WWE} and is \cref{thm:RisomT} for $M = N^2$.  

\begin{proposition}\label{prop:isomodI}
    Assume that $\Phi_{\Gamma_0(M)}$ is an isomorphism.  Then $\Phi_{\Gamma_1(M)}$ induces isomorphisms 
    \[
    R_{\Gamma_0(M)} \cong R_{\Gamma_1(M)}/\II R_{\Gamma_1(M)} \cong \TT_{\Gamma_1(M)}/\II \TT_{\Gamma_1(M)} \cong \TT_{\Gamma_0(M)}.
    \]
\end{proposition}

\begin{proof}
    Note that $\II$ acts trivially on $\TT_{\Gamma_0(M)}$ since $\det(D_{\Gamma_0(M)}) = \varepsilon$, so there is a natural surjection from  $\TT_{\Gamma_1(M)}/\II \TT_{\Gamma_1(M)}$ to $\TT_{\Gamma_0(M)}$.  Combining this with \cref{lem:ZpDeltAlg} and the fact that $\Phi_{\Gamma_1(M)}$ is surjective gives a sequence of surjections
    \[
    R_{\Gamma_0(M)} \cong R_{\Gamma_1(M)}/\II R_{\Gamma_1(M)} \twoheadrightarrow \TT_{\Gamma_1(M)}/\II \TT_{\Gamma_1(M)} \twoheadrightarrow \TT_{\Gamma_0(M)}.
    \]
    This composition is $\Phi_{\Gamma_0(M)}$, which is an isomorphism by assumption.  Thus all of the morphisms are isomorphisms.
\end{proof}

\section{The \texorpdfstring{$\Z_p[\Delta]^+$}{}-module structure}\label{sec:modrepthy}

A key observation for our main results, established in \cref{subsec:ZpDeltaPlus}, is that $R_{\Gamma_0(N^2)}$ (and hence $\TT_{\Gamma_0(N^2)}$) is a $\Z_p[\Delta]^+$-algebra, where $+$ denotes the subring fixed under the involution that inverts group-like elements. In \cref{subsec:relationGamma0N2Gamma1N}, we  prove \cref{thmA:modularity} and \cref{prop:freenessequiv}, thereby relating freeness results and modularity theorems across levels $\Gamma_0(N^2)$ and $\Gamma_1(N)$. In \cref{subsec:R=Tfinalstep}, we prove \cref{thmA:R=TGamma0N2,thmA:Gamma1N2+rank} from the freeness result of \cite{LangPollackWake}, thus establishing modularity at levels $\Gamma_0(N^2)$ and $\Gamma_1(N^2)$.  Modularity for $\Gamma_1(N)$ follows from \cref{thmA:R=TGamma0N2} and the equivalence of \cref{thmA:modularity}.

The idea in  \cref{subsec:relationGamma0N2Gamma1N,subsec:R=Tfinalstep} is to apply the commutative algebra  \cref{lem:NAKReduction,lem:NAKpluscounting} to three situations:
\begin{itemize}
\item[(i)] the map $\varphi=\Phi_{\Gamma_0(N^2)}: R_{\Gamma_0(N^2)} \rightarrow \TT_{\Gamma_0(N^2)}$ with $A = \Z_p[\Delta]^+$, 
\item[(ii)] the map $\varphi=\Phi_{\Gamma_1(N)}: R_{\Gamma_1(N)} \rightarrow \TT_{\Gamma_1(N)}$ with $A = \Z_p[\Delta]$, and 
\item[(iii)] the map $\varphi=\Phi_{\Gamma_1(N^2)}: R_{\Gamma_1(N^2)} \rightarrow \TT_{\Gamma_1(N^2)}$ with $A = \Z_p[\Delta]$. 
\end{itemize}To verify the hypotheses of \cref{lem:NAKReduction,lem:NAKpluscounting} in each situation, one uses counting arguments for ranks of Hecke algebras from \cref{subsec:countingarg} and results  from \cref{subsec:redstep} to reduce to a situation where modularity is known. For levels $\Gamma_0(N^2)$ and $\Gamma_1(N)$, the rank counts are given by \cref{cor:countingranksequivalence}, while the reduction to the modularity theorem at level $\Gamma_0(N)$ \cite[Corollary 7.1.3]{WWE} is done using \cref{prop:isomodI,prop:modaugideal}, respectively. For level $\Gamma_1(N^2)$, \cref{prop:countingranksN2} gives the rank count, and the reduction to modularity theorem at level $\Gamma_0(N^2)$ (\cref{thmA:R=TGamma0N2}) is done with \cref{prop:isomodI}.


\subsection{The \texorpdfstring{$\Z_p[\Delta]^+$}{}-algebra structure on \texorpdfstring{$R_{\Gamma_0(N^2)}$}{}}\label{subsec:ZpDeltaPlus}
In this section we construct a natural $\Z_p$-algebra homomorphism $\Z_p[\Delta]^+ \to R_{\Gamma_0(N^2)}$.  The idea is to identify $\Z_p[\Delta]^+$ as an inertia-at-$N$ pseudodeformation ring, which then admits a map to $R_{\Gamma_0(N^2)}$ via restriction and universality.  This is nearly identical to \cite[\S 4.2]{LangWake25}, though there one assumes that $N \equiv -1 \bmod p$.

Let $\tilde{R}_N$ be the universal pseudodeformation ring of the trivial 2-dimensional pseudorepresentation $\psi(1 \oplus 1) \colon I_N \to \F_p$ parametrizing deformations with trivial determinant.  Any such pseudodeformation factors through the maximal pro-$p$ quotient of $I_N$, which we denote by $I_N^{(p)} \cong \Z_p$ \cite[Lemma 7.64]{Chenevier}.  Write $\tilde{D}_N \colon I_N^{(p)} \to \tilde{R}_N$ for the corresponding universal pseudorepresentation.  Fix a topological generator $\tau \in I_N^{(p)}$.  Recall that
\[
\Frob_N \tau \Frob_N^{-1} = \tau^N,
\]
and hence any pseudorepresentation of $I_N^{(p)}$ that extends to a pseudorepresentation on $G_{\Q_N}$ necessarily has the same trace on $\tau$ and $\tau^N$.  Define $R_N \coloneqq \tilde{R}_N/(\tr(\tilde{D}_N)(\tau) - \tr(\tilde{D}_N)(\tau^N))$.  Let $D_N \colon I_N^{(p)} \to R_N$ be the corresponding pseudorepresentation obtained by following $D_N$ with the natural quotient map $\tilde{R}_N \twoheadrightarrow R_N$.  Note that $\Delta$ is naturally a quotient of $I_N^{(p)}$ since it is the Galois group of the maximal $p$-power subextension of the totally ramified extension $\Q_N(\zeta_N)/\Q_N$.  Thus the choice of generator $\tau \in I_N^{(p)}$ determines a choice of generator for $\Delta$, which we denote by $\delta_{\tau}$.

\begin{proposition}\label{prop:ZpDeltaPlusStructure}
The $\Z_p$-algebra $R_N$ is isomorphic to $\Z_p[\Delta]^+$ via a map that identifies $\tr(D_N)(\tau) \in R_N$ with $[\delta_\tau] + [\delta_\tau^{-1}] \in \Z_p[\Delta]^+$.
\end{proposition}

\begin{proof}
The proof is nearly identical to \cite[Proposition 4.4]{LangWake25} with minor tweak since the congruence condition on $N$ is different.  We summarize the strategy and point out the change but leave the reader to consult \cite{LangWake25} for the details.

First one shows that $\tilde{R}_N \cong \Z_p\lb x \rb$ via an isomorphism that identifies $\tr(\tilde{D}_N)(\tau)$ and $2 + x$. To compute the quotient by the ideal generated by $\tr(\tilde{D}_N)(\tau) - \tr(\tilde{D}_N)(\tau^N)$, one adjoins the roots of the characteristic polynomial of $\tilde{D}_N(\tau)$ to $\Z_p\lb x \rb$.  That is, we embed $\Z_p\lb x \rb$ into $\frac{\Z_p\lb x \rb[\lambda]}{\lambda^2 - (2 + x)\lambda + 1} \cong \Z_p\lb \lambda - 1 \rb$.  Note that the element $x$ corresponds to $\frac{(\lambda - 1)^2}{\lambda}$ in $\Z_p\lb \lambda - 1 \rb$.  Let $A$ be the quotient of $\Z_p\lb \lambda - 1 \rb$ by the ideal generated by 
\[
\tr(\tilde{D}_N)(\tau) - \tr(\tilde{D}_N)(\tau^N) = \lambda + \lambda^{-1} - \lambda^N - \lambda^{-N} = -\lambda^{-N}(\lambda^{N+1} - 1)(\lambda^{N-1} - 1).
\]

Since $p \nmid N+1$, we find that 
\[
\frac{\lambda^{N+1} - 1}{\lambda - 1}, \frac{\lambda^{N-1} - 1}{\lambda^{p^s} - 1} \in \Z_p\lb \lambda - 1 \rb^\times.
\]
(The roles of $N+1$ and $N-1$ are reversed in \cite[Proposition 4.4]{LangWake25}.) Thus 
\[
A = \Z_p\lb \lambda - 1 \rb/(\lambda + \lambda^{-1} - \lambda^N - \lambda^{-N})  = \Z_p\lb \lambda - 1 \rb/(\lambda - 1)(\lambda^{p^s} - 1),
\]
and $R_N$ is isomorphic to the $\Z_p$-subalgebra of $A$ generated by the image of $x$, namely $\frac{(\lambda - 1)^2}{\lambda}$.  The surjection $A \twoheadrightarrow \Z_p[\Delta]$ given by sending $\lambda$ to $\delta_\tau$ sends $\frac{(\lambda - 1)^2}{\lambda}$ to $[\delta_\tau] + [\delta_\tau^{-1}] - 2$.  Thus $R_N$ projects to $\Z_p[\Delta]^+$ under $A \twoheadrightarrow \Z_p[\Delta]$ since $\Z_p[\Delta]^+$ is the $\Z_p$-subalgebra generated by $[\delta_\tau] + [\delta_\tau^{-1}] - 2$.  That this  is an isomorphism follows as in \cite[Proposition 4.4]{LangWake25}.

Finally, the identification of $\tr(\tilde{D}_N)(\tau)$ with $2 + x$ becomes an identification of $\tr(D_N)(\tau)$ with $2 + \frac{(\lambda - 1)^2}{\lambda} = \lambda + \lambda^{-1}$ in $R_N \subseteq A$.  The quotient of $A$ given by sending $\lambda$ to $\delta_\tau$ then identifies $\lambda + \lambda^{-1}$ with $[\delta_\tau] + [\delta_\tau^{-1}] \in \Z_p[\Delta]^+$, as desired.
\end{proof}

\begin{corollary}\label{cor:ZpDeltaPlusAlg}
The ring $R_{\Gamma_0(N^2)}$ is a $\Z_p[\Delta]^+$-algebra.
\end{corollary}

\begin{proof}
The universal deformation $D = D_{\Gamma_0(N^2)}$ restricts to a deformation of the trivial pseudorepresentation on $I_N$ with trivial determinant.  This deformation on $I_N$ extends to one on $G_{\Q_N}$ since $D$ is defined on $G_{\Q, S}$.  Thus $D|_{I_N}$ gives rise to a $\Z_p$-algebra homomorphism $R_N \to R_{\Gamma_0(N^2)}$ by universality, which gives the desired structure by \cref{prop:ZpDeltaPlusStructure}.
\end{proof}

Combining \cref{cor:ZpDeltaPlusAlg} with the map $R_{\Gamma_0(N^2)} \to \TT_{\Gamma_0(N^2)}$ makes $\TT_{\Gamma_0(N^2)}$ into a $\Z_p[\Delta]^+$-algebra.  Note that the trivial $\Z_p$-valued pseudorepresentation gives rise to an algebra homomorphism $\Z_p[\Delta]^+ = R_N \to \Z_p$ by universality, which is the usual augmentation map.  Let $\II^+$ be its kernel, which is the augmentation ideal of $\Z_p[\Delta]^+$.

\begin{proposition}\label{prop:modaugideal}
The natural map $\Phi_{\Gamma_0(N^2)} \colon R_{\Gamma_0(N^2)} \to \TT_{\Gamma_0(N^2)}$ induces isomorphisms
\[
R_{\Gamma_0(N)} \cong R_{\Gamma_0(N^2)}/\II^+ R_{\Gamma_0(N^2)} \cong \TT_{\Gamma_0(N^2)}/\II^+\TT_{\Gamma_0(N^2)} \cong \TT_{\Gamma_0(N)}.
\]
\end{proposition}

\begin{proof}
Let $\alpha \colon R_{\Gamma_0(N^2)} \to A$ be a $\Z_p$-algebra homomorphism with corresponding deformation $D_\alpha \colon G_{\Q,S} \to A$.  Note that $\alpha$ factors through $R_{\Gamma_0(N^2)}/\II^+ R_{\Gamma_0(N^2)}$ if and only if $D_\alpha|_{I_N}$ is the trivial pseudorepresentation.  Thus we see that $R_{\Gamma_0(N^2)}/\II^+ R_{\Gamma_0(N^2)}$ is the quotient of $R_{\Gamma_0(N^2)}$ parametrizing pseudodeformations that are pseudotrivial when restricted to $I_N$, which is also the definition of $R_{\Gamma_0(N)}$.  Thus $R_{\Gamma_0(N)} \cong R_{\Gamma_0(N^2)}/\II^+ R_{\Gamma_0(N^2)}$.

The map $\TT_{\Gamma_0(N^2)} \to \TT_{\Gamma_0(N)}$ given in \cref{prop:mapfromGamma0N2toGamma0N} factors through $\TT_{\Gamma_0(N^2)}/\II^+\TT_{\Gamma_0(N^2)}$.  Indeed, we can view
\[
\TT_{\Gamma_0(N)} \subseteq \prod_{f \in \mathcal{F}_{\Gamma_0(N)}} \overline{\Z}_p.
\]
Thus it suffices to show that $D_f|_{I_N}$ is trivial for any $f \in \mathcal{F}_{\Gamma_0(N)}$.  This is true for cusp forms since they are Steinberg at $N$.  For $f = E_{2,N}$ this is true since its associated pseudorepresentation is $\psi(\omega \oplus 1)$, which is unramified at $N$.

Thus $\Phi_{\Gamma_0(N^2)}$ induces a series of surjections $R_{\Gamma_0(N)} \cong R_{\Gamma_0(N^2)}/\II^+ R_{\Gamma_0(N^2)} \twoheadrightarrow \TT_{\Gamma_0(N^2)}/\II^+\TT_{\Gamma_0(N^2)} \twoheadrightarrow \TT_{\Gamma_0(N)}$.
This map is $\Phi_{\Gamma_0(N)}$, which is an isomorphism \cite[Corollary 7.1.3]{WWE}.  Thus all of the maps are isomorphisms, as desired.
\end{proof}


\subsection{Relating freeness and modularity across levels \texorpdfstring{$\Gamma_0(N^2)$}{} and \texorpdfstring{$\Gamma_1(N)$}{}}\label{subsec:relationGamma0N2Gamma1N}

\begin{lemma}\label{lem:Deltaplusfreefromrank}
 The  $\Z_p[\Delta]^+$-module $\TT_{\Gamma_0(N^2)}$ is free of rank $r$ if and only if $\rk_{\Z_p} \TT_{\Gamma_0(N^2)}= \frac{(p^s+1)}{2}r$.
\end{lemma}
\begin{proof}
The forward implication is straightforward, so we focus on proving the reverse implication. Suppose $\rk_{\Z_p} \TT_{\Gamma_0(N^2)}= \frac{(p^s+1)}{2}r$. Recall that $\II^+$ denotes the kernel of the natural augmentation map $\Z_p[\Delta]^+ \rightarrow \Z_p$. By \cref{prop:modaugideal}, we have $\TT_{\Gamma_0(N^2)}/\II^+ \TT_{\Gamma_0(N^2)} \cong \TT_{\Gamma_0(N)}$. Thus, $\rk_{\Z_p} \TT_{\Gamma_0(N^2)}/\II^+ \TT_{\Gamma_0(N^2)}=r$. Consequently, we obtain the following equality:
\begin{equation}\label{eq:Deltaplusfreerankcomp}
\underbrace{\rk_{\Z_p} \TT_{\Gamma_0(N^2)}}_{\frac{(p^s+1)}{2}r}= \underbrace{\left(\rk_{\Z_p} \Z_p[\Delta]^+\right)}_{\frac{p^s+1}{2}} \underbrace{\left(\rk_{\Z_p} \TT_{\Gamma_0(N^2)}/\II^+ \TT_{\Gamma_0(N^2)}\right)}_{r}.
\end{equation}
The lemma now follows from \eqref{eq:Deltaplusfreerankcomp} and \cref{lem:NAKpluscounting} since $\TT_{\Gamma_0(N^2)}$ and $\TT_{\Gamma_0(N)}$ are free $\Z_p$-modules.
\end{proof}

\begin{lemma} \label{lem:Deltafreefromrank}
The $\Z_p[\Delta]$-module $\TT_{\Gamma_1(N)}$ is free  of rank $r$ if and only if $\rk_{\Z_p} \TT_{\Gamma_1(N)}=p^sr$.
\end{lemma}
\begin{proof}
The forward implication is straightforward, so we focus on proving the reverse implication. Suppose $\rk_{\Z_p} \TT_{\Gamma_1(N)}=p^sr$. Recall $\II$ denotes the kernel of the natural augmentation map $\Z_p[\Delta] \rightarrow \Z_p$. By combining \cref{prop:isomodI} with the modularity theorem of Wake--Wang-Erickson \cite[Corollary 7.1.3]{WWE}, we have $\TT_{\Gamma_1(N)}/\II \TT_{\Gamma_1(N)}  \cong \TT_{\Gamma_0(N)}$. Thus, $\rk_{\Z_p} \TT_{\Gamma_1(N)}/\II \TT_{\Gamma_1(N)} = \rk_{\Z_p} \TT_{\Gamma_0(N)}= r$. Consequently, we obtain the following equality:
\begin{equation}\label{eq:Deltafreerankcomp}
\underbrace{\rk_{\Z_p} \TT_{\Gamma_1(N)}}_{p^sr}= \underbrace{\left(\rk_{\Z_p} \Z_p[\Delta]\right)}_{p^s} \underbrace{\left(\rk_{\Z_p} \TT_{\Gamma_1(N)}/\II \TT_{\Gamma_1(N)}\right)}_{r}.
\end{equation}
The lemma now follows from \eqref{eq:Deltafreerankcomp} and \cref{lem:NAKpluscounting} since $\TT_{\Gamma_1(N)}$ and $\TT_{\Gamma_0(N)}$ are free $\Z_p$-modules.
\end{proof}

We restate \cref{prop:freenessequiv} from \cref{sec:introduction} and provide its proof. 
\begin{proposition}
    The following statements are equivalent:
    \begin{enumerate}[label=(\roman*),leftmargin=*]
    \item\label{item:Deltaplusfreeness} the $\Z_p[\Delta]^+$-module $\TT_{\Gamma_0(N^2)}$ is free of rank $r$;
    \item\label{item:Deltafreeness} the $\Z_p[\Delta]$-module $\TT_{\Gamma_1(N)}$ is free of rank $r$.
    \end{enumerate}
\end{proposition}

\begin{proof}
By \cref{lem:Deltaplusfreefromrank} statement \ref{item:Deltaplusfreeness} is equivalent to $\rk_{\Z_p}\TT_{\Gamma_0(N^2)}=\frac{p^s+1}{2}r$. By \cref{cor:countingranksequivalence} this is equivalent to $\rk_{\Z_p} \TT_{\Gamma_1(N)}=p^sr$, which is in turn equivalent to statement \ref{item:Deltafreeness} by \cref{lem:Deltafreefromrank}.
\end{proof}

We restate \cref{thmA:modularity} from \cref{sec:introduction} and provide its proof. 
\begin{theorem}
The following statements are equivalent:
\begin{enumerate}[label=(\roman*),leftmargin=*]
    \item\label{thmitem:Gamma0N2+rank} $R_{\Gamma_0(N^2)} \cong \TT_{\Gamma_0(N^2)}$ and $\rk_{\Z_p} \TT_{\Gamma_0(N^2)} = \bigl(\frac{p^s + 1}{2}\bigr)r$;
    \item\label{thmitem:Gamma1N+rank}   $R_{\Gamma_1(N)} \cong \TT_{\Gamma_1(N)}$  and $\rk_{\Z_p} \TT_{\Gamma_1(N)} = p^sr$.
    \end{enumerate}
\end{theorem}

\begin{proof}
To prove that \ref{thmitem:Gamma0N2+rank} implies \ref{thmitem:Gamma1N+rank}, combine the equation $\rk_{\Z_p} \TT_{\Gamma_0(N^2)}= \frac{(p^s+1)}{2}r$ of \ref{thmitem:Gamma0N2+rank} with \cref{cor:countingranksequivalence}, to get $\rk_{\Z_p} \TT_{\Gamma_1(N)}=p^sr$. This allows us to apply \cref{lem:Deltafreefromrank} to deduce that $\TT_{\Gamma_1(N)}$ is a free $\Z_p[\Delta]$-module of rank $r$. 
Combining \cref{prop:isomodI} with the modularity theorem in level $\Gamma_0(N)$ \cite[Corollary 7.1.3]{WWE}, we have  $R_{\Gamma_1(N)}/\II R_{\Gamma_1(N)}  \cong \TT_{\Gamma_1(N)}/\II \TT_{\Gamma_1(N)}$. Now \ref{thmitem:Gamma1N+rank} follows by applying \cref{lem:NAKReduction} with $A = \Z_p[\Delta]$ and $\varphi = \Phi_{\Gamma_1(N)} \colon R_{\Gamma_1(N)} \to \TT_{\Gamma_1(N)}$.

To see that \ref{thmitem:Gamma1N+rank} implies \ref{thmitem:Gamma0N2+rank}, combine $\rk_{\Z_p} \TT_{\Gamma_1(N)} = p^sr$ with \cref{cor:countingranksequivalence} to get $\rk_{\Z_p} \TT_{\Gamma_0(N^2)}= \frac{(p^s+1)}{2}r$. Thus we may apply \cref{lem:Deltaplusfreefromrank} to see that $\TT_{\Gamma_0(N^2)}$ is a free $\Z_p[\Delta]^+$-module of rank $r$. 
\cref{prop:modaugideal} implies $R_{\Gamma_0(N^2)}/\II^+ R_{\Gamma_0(N^2)}  \cong \TT_{\Gamma_0(N^2)}/\II^+ \TT_{\Gamma_0(N^2)}$. Now \ref{thmitem:Gamma0N2+rank} follows by applying \cref{lem:NAKReduction} with $A = \Z_p[\Delta]^+$ and $\varphi = \Phi_{\Gamma_0(N^2)}$.
\end{proof}

\subsection{Proving the modularity theorems from the freeness result at level \texorpdfstring{$\Gamma_0(N^2)$}{}}\label{subsec:R=Tfinalstep}
We complete the proofs of the modularity theorems stated in \cref{sec:introduction}. The crucial fact we need is the following freeness result of the first author with Pollack and Wake. 

\begin{theorem}\cite{LangPollackWake}\label{thm:TisFree}
The $\Z_p[\Delta]^+$-module $\TT_{\Gamma_0(N^2)}$ is free of rank $r$.
\end{theorem}

We now restate \cref{thmA:R=TGamma0N2} from \cref{sec:introduction} and provide its proof. 
\begin{theorem}\label{thm:RisomT}
We have  $R_{\Gamma_0(N^2)} \cong \TT_{\Gamma_0(N^2)}$  and  $\rk_{\Z_p} \TT_{\Gamma_0(N^2)} = \bigl(\frac{p^s + 1}{2}\bigr)r$.
\end{theorem}

\begin{proof}
The  assertion $\rk_{\Z_p} \TT_{\Gamma_0(N^2)}= (\frac{p^s + 1}{2})r$ follows from \cref{thm:TisFree}. To prove that $R_{\Gamma_0(N^2)} \cong \TT_{\Gamma_0(N^2)}$, apply \cref{lem:NAKReduction} with $A = \Z_p[\Delta]^+$ and $\varphi = \Phi_{\Gamma_0(N^2)} \colon R_{\Gamma_0(N^2)} \to \TT_{\Gamma_0(N^2)}$.  \cref{thm:TisFree} shows that $\TT$ is a free $\Z_p[\Delta]^+$-module of rank $r$, and \cref{prop:modaugideal} shows that $\Phi_{\Gamma_0(N^2)}$ induces an isomorphism modulo the augmentation ideal $\II^+$.
\end{proof}

We now restate \cref{thmA:Gamma1N2+rank} from \cref{sec:introduction} and provide its proof.

\begin{theorem} 
 We have $R_{\Gamma_1(N^2)} \cong \TT_{\Gamma_1(N^2)}$  and $\rk_{\Z_p} \TT_{\Gamma_1(N^2)} = \bigl(\frac{p^s(p^s + 1)}{2}\bigr)r$.
\end{theorem}
\begin{proof}
The conclusion  $R_{\Gamma_0(N^2)} \cong \TT_{\Gamma_0(N^2)}$ of \cref{thm:RisomT} allows us to deduce that $\Phi_{\Gamma_1(N^2)}$ induces an isomorphism modulo the augmentation ideal $\II$ of $\Z_p[\Delta]$ from \cref{prop:isomodI}. Thus we have $\rk_{\Z_p}\left(\TT_{\Gamma_1(N^2)}/\II\TT_{\Gamma_1(N^2)}\right) = \rk_{\Z_p}\TT_{\Gamma_0(N^2)}=\left(\frac{p^s+1}{2}\right)r$. By \cref{prop:countingranksN2}, we have $\rk_{\Z_p} \TT_{\Gamma_1(N^2)}=p^s\left(\frac{p^s+1}{2}\right)r$. Consequently, we obtain the following equality:
\begin{equation}
\underbrace{\rk_{\Z_p} \TT_{\Gamma_1(N^2)}}_{p^s\left(\frac{p^s+1}{2}\right)r}= \underbrace{\left(\rk_{\Z_p} \Z_p[\Delta]\right)}_{p^s} \underbrace{\left(\rk_{\Z_p} \TT_{\Gamma_1(N^2)}/\II \TT_{\Gamma_1(N^2)}\right)}_{\frac{(p^s+1)}{2}r}.
\end{equation}
Since $\TT_{\Gamma_1(N^2)}$ and $\TT_{\Gamma_0(N^2)}$ are free $\Z_p$-modules, \cref{lem:NAKpluscounting} now lets us conclude that $\TT_{\Gamma_1(N^2)}$  is a free $\Z_p[\Delta]$-module. The theorem now follows by applying \cref{lem:NAKReduction} with $A = \Z_p[\Delta]$ with $\varphi = \Phi_{\Gamma_1(N^2)} \colon R_{\Gamma_1(N^2)} \to \TT_{\Gamma_1(N^2)}$. 
\end{proof}


\begin{remark}\label{rem:LWWE}

There is another approach, which we learned from unpublished work of Lecouturier--Wake--Wang-Erickson, that one could take to proving the freeness results of \cref{prop:freenessequiv} and the modularity theorems $R_\Gamma \cong \TT_\Gamma$ for all four levels $\Gamma$. It involves showing directly that $\rk_{\Z_p} \TT_{\Gamma_1(N)} = |\Delta|\rk_{\Z_p} \TT_{\Gamma_0(N)}$.  This can be done before localizing at the Eisenstein maximal ideal in the sense that one can show that $M_2(\Gamma_1(6N); \overline{\Q}_p)$ has dimension equal to the product of $|(\Z/6N\Z)^\times/\{\pm 1\}|$ with the dimension of $M_2(\Gamma_0(6N); \overline{\Q}_p)$.  The idea of Lecouturier--Wake--Wang-Erickson is that for $\Gamma = \Gamma_0(6N)$ or $\Gamma_1(6N)$, the dimension of $M_2(\Gamma; \overline{\Q}_p)$ is equal to $g(\Gamma) - 1 + e_\infty(\Gamma)$, where $g(\Gamma)$ is the genus of the modular curve $X(\Gamma)$ and $e_\infty(\Gamma)$ is the number of cusps on $X(\Gamma)$.  One checks that both $g - 1$ and $e_\infty$ are multiplicative in the degree, and $X_1(6N)$ is a $(\Z/6N\Z)^\times/\{\pm 1\}$-cover of $X_0(6N)$.  The result follows from this, and localizing at an appropriate Eisenstein maximal ideal can be used to get rid of the auxiliary 6 and conclude that $\rk_{\Z_p} \TT_{\Gamma_1(N)} = |\Delta|\rk_{\Z_p} \TT_{\Gamma_0(N)}$; this is the latter assertion of \cref{thmA:modularity}\ref{thmpart:Gamma1N+rank}. This counting result immediately implies that the $\Z_p[\Delta]$-module $\TT_{\Gamma_1(N)}$ is free of rank $r$ by \cref{lem:Deltafreefromrank}.  \cref{thmA:modularity}\ref{thmpart:Gamma1N+rank} follows  by applying \cref{lem:NAKReduction} with $A = \Z_p[\Delta]$ with $\varphi = \Phi_{\Gamma_1(N)} \colon R_{\Gamma_1(N)} \to \TT_{\Gamma_1(N)}$.  Using the equivalence given in \cref{thmA:modularity} leads us to \cref{thmA:R=TGamma0N2} proving $R_{\Gamma_0(N^2)} \cong \TT_{\Gamma_0(N^2)}$, from which one then obtains $R_{\Gamma_1(N^2)} \cong \TT_{\Gamma_1(N^2)}$ in \cref{thmA:Gamma1N2+rank}. It is easy to see that this argument no longer works when $p = 2, 3$.

\end{remark}

\section*{Acknowledgements}
This project started at the ``Pair of Automorphic Workshops" at the University of Oregon in 2022, supported by the National Science Foundation grant DMS-1751281 and the National Security Agency MSP conference grant H98230-21-1-0029.  The authors thank Ellen Eischen and the other organizers for facilitating the workshop.  They also thank Romyar Sharifi, who contributed to the early stages of this project through the aforementioned workshop.  The authors thank the following people for helpful conversations: Shaunak Deo, David Helm, Pedro Lemos, Robert Pollack, Preston Wake, and Carl Wang-Erickson.  Their thoughts have improved this paper. J.L.~acknowledges support from the National Science Foundation through the grant DMS-2301738 and from the Simons Foundation through the award MP-TSM-00002260. B.P.'s research is partially supported by the Infosys Young Investigator Award, from the Infosys Foundation Bangalore, the SERB-MATRICS grant MTR/2022/000244, and DST FIST program 2021 [TPN - 700661].

\bibliography{modularity}

\providecommand{\bysame}{\leavevmode\hbox to3em{\hrulefill}\thinspace}
\providecommand{\MR}{\relax\ifhmode\unskip\space\fi MR }
\providecommand{\MRhref}[2]{%
  \href{http://www.ams.org/mathscinet-getitem?mr=#1}{#2}
}
\providecommand{\href}[2]{#2}
\begin{thebibliography}{WWE20}

\bibitem[Car89]{Carayol89}
Henri Carayol, \emph{Sur les repr\'esentations galoisiennes modulo {$l$}
  attach\'ees aux formes modulaires}, Duke Math. J. \textbf{59} (1989), no.~3,
  785--801. \MR{1046750}

\bibitem[CE05]{CE2005}
Frank Calegari and Matthew Emerton, \emph{On the ramification of {H}ecke
  algebras at {E}isenstein primes}, Invent. Math. \textbf{160} (2005), no.~1,
  97--144. \MR{2129709}

\bibitem[Che14]{Chenevier}
Ga\"etan Chenevier, \emph{The {$p$}-adic analytic space of pseudocharacters of
  a profinite group and pseudorepresentations over arbitrary rings},
  Automorphic forms and {G}alois representations. {V}ol. 1, London Math. Soc.
  Lecture Note Ser., vol. 414, Cambridge Univ. Press, Cambridge, 2014,
  pp.~221--285. \MR{3444227}

\bibitem[DS05]{DiamondShurman}
Fred Diamond and Jerry Shurman, \emph{A first course in modular forms},
  Graduate Texts in Mathematics, vol. 228, Springer-Verlag, New York, 2005.
  \MR{2112196}

\bibitem[Lec21]{Lecouturier}
Emmanuel Lecouturier, \emph{Higher {E}isenstein elements, higher {E}ichler
  formulas and rank of {H}ecke algebras}, Invent. Math. \textbf{223} (2021),
  no.~2, 485--595. \MR{4209860}

\bibitem[LMP26]{LMP}
Jaclyn Lang, Katharina Müller, and Bharathwaj Palvannan, \emph{A new
  prespective on the rank of mazur's eisenstein ideal},
  \url{arxiv.org/pdf/2605.04195}, 2026, preprint.

\bibitem[LPW26]{LangPollackWake}
Jaclyn Lang, Robert Pollack, and Preston Wake, \emph{Counting level-raising
  congruences using modular representation theory}, 2026.

\bibitem[LW12]{LoefflerWeinstein}
David Loeffler and Jared Weinstein, \emph{On the computation of local
  components of a newform}, Math. Comp. \textbf{81} (2012), no.~278,
  1179--1200. \MR{2869056}

\bibitem[LW25]{LangWake25}
Jaclyn Lang and Preston Wake, \emph{The {E}isenstein ideal at prime-square
  level has constant rank}, Proc. Natl. Acad. Sci. USA \textbf{122} (2025),
  no.~28, Paper No. e2500729122.

\bibitem[Maz77]{Mazur}
B.~Mazur, \emph{Modular curves and the {E}isenstein ideal}, Inst. Hautes
  \'Etudes Sci. Publ. Math. (1977), no.~47, 33--186, With an appendix by Mazur
  and M. Rapoport. \MR{488287}

\bibitem[Mer96]{Merel}
Lo\"ic Merel, \emph{L'accouplement de {W}eil entre le sous-groupe de {S}himura
  et le sous-groupe cuspidal de {$J_0(p)$}}, J. Reine Angew. Math. \textbf{477}
  (1996), 71--115. \MR{1405312}

\bibitem[Miy06]{Miyakebook}
Toshitsune Miyake, \emph{Modular forms}, english ed., Springer Monographs in
  Mathematics, Springer-Verlag, Berlin, 2006, Translated from the 1976 Japanese
  original by Yoshitaka Maeda. \MR{2194815}

\bibitem[MW86]{MazurWiles86}
B.~Mazur and A.~Wiles, \emph{On {$p$}-adic analytic families of {G}alois
  representations}, Compositio Math. \textbf{59} (1986), no.~2, 231--264.
  \MR{860140}

\bibitem[NSW08]{Neukirch}
J\"{u}rgen Neukirch, Alexander Schmidt, and Kay Wingberg, \emph{Cohomology of
  number fields}, second ed., Grundlehren der Mathematischen Wissenschaften,
  vol. 323, Springer-Verlag, Berlin, 2008.

\bibitem[Shi71]{Shimura}
Goro Shimura, \emph{Introduction to the arithmetic theory of automorphic
  functions}, Kan\^o{} Memorial Lectures, vol. No. 1, Iwanami Shoten
  Publishers, Tokyo; Princeton University Press, Princeton, NJ, 1971,
  Publications of the Mathematical Society of Japan, No. 11. \MR{314766}

\bibitem[SW99]{SkinnerWiles99}
C.~M. Skinner and A.~J. Wiles, \emph{Residually reducible representations and
  modular forms}, Inst. Hautes \'Etudes Sci. Publ. Math. (1999), no.~89,
  5--126. \MR{1793414}

\bibitem[WWE19]{WWE2019}
Preston Wake and Carl Wang-Erickson, \emph{Deformation conditions for
  pseudorepresentations}, Forum Math. Sigma \textbf{7} (2019), Paper No. e20,
  44. \MR{3987305}

\bibitem[WWE20]{WWE}
\bysame, \emph{The rank of {M}azur's {E}isenstein ideal}, Duke Math. J.
  \textbf{169} (2020), no.~1, 31--115. \MR{4047548}

\end{thebibliography}
\bibliographystyle{amsalpha}

\end{document}